\documentclass[10pt]{article}
\usepackage{latexsym, amsmath, amssymb, amsthm}
\usepackage{enumerate, mathrsfs}
\usepackage{graphicx, tikz}
\usepackage{authblk, setspace, floatrow, cite}
\usepackage{hyperref}
\usepackage[top=0.8 in, bottom=1 in, left=1.1 in, right=1.1 in]{geometry}
\newtheorem{thm}{Theorem}[section]
\newtheorem{lem}[thm]{Lemma}
\newtheorem{defi}[thm]{Definition}

\newtheorem{pro}[thm]{Proposition}

\makeatletter
\def\blfootnote{\xdef\@thefnmark{}\@footnotetext}
\makeatother
\newfloatcommand{capbtabbox}{table}[][\FBwidth]

\numberwithin{equation}{section}

\newcommand{\R}{\mathbb R}

\def\m{\mathbb}    \def\mb{\mathbf}		
\def\eps{\epsilon}  \def\a{\alpha}  \def\b{\beta}  \def\th{\theta}	   \def\si {\sigma}
  	 \def\ld{\lambda}
\def\la{\langle}  \def\ra{\rangle}  \def\wh{\widehat}
\def\l{\left}  \def\r{\right}

\def\be{\begin{equation}}     \def\ee{\end{equation}}
\def\bp{\begin{pmatrix}}	\def\ep{\end{pmatrix}}

\def\H{\mathcal{H}}

\def\X{\mathcal{X}}

\date{}

\begin{document}

\title{Lower Regularity Solutions of the Non-homogeneous Boundary-Value  Problem for a Higher Order Boussinesq Equation in a Quarter Plane}

\author{Shenghao Li\footnote{School of Mathematical Sciences,University of Electronic Science and Technology of China, Chengdu, PR China, email: lish@uestc.edu.cn}, 
	Min Chen\footnote{Department of Mathematics,  Purdue University, West Lafayette, IN, USA, email: chen@math.purdue.edu}, 
	Xin Yang\footnote{Department of Mathematics, University of California at Riverside,  Riverside, CA 92521, USA, email: xiny@ucr.edu},  
	Bing-Yu Zhang\footnote{ Department of Mathematical Sciences, University of Cincinnati, Cincinnati, OH 45221, USA, email: zhangb@ucmail.uc.edu}}
\date{}
\maketitle

\begin{abstract}
We continue to study the  initial-boundary-value problem of the sixth order Boussinesq equation in a quarter plane
with non-homogeneous boundary conditions:
\begin{equation*}
\begin{cases}
u_{tt}-u_{xx}+\beta u_{xxxx}-u_{xxxxxx}+(u^2)_{xx}=0,\quad x,t\in \mathbb{R}^+,\\
u(x,0)=\varphi (x), u_t(x,0)=\psi ''(x),  \\
u(0,t)=h_1(t), u_{xx}(0,t)=h_2(t), u_{xxxx}(0,t)=h_3(t),
\end{cases}
\end{equation*}
where $\beta=\pm1$. We  show that the problem is  locally analytically well-posed  in  the space $H^s(\mathbb{R}^+)$ for   any $ s> -\frac34 $ with the initial-value data $$(\varphi,\psi)\in H^s(\mathbb{R}^+)\times H^{s-1}(\mathbb{R}^+)$$ and  the boundary-value data 
$$(h_1,h_2,h_3) \in H^{\frac{s+1}{3}}(\mathbb{R}^+)\times H^{\frac{s-1}{3}}(\mathbb{R}^+)\times H^{\frac{s-3}{3}}(\mathbb{R}^+).$$
\end{abstract}





\section{Introduction}

\qquad

In this article, we continue the study  for  the initial-boundary-value problem of the sixth order Boussinesq equation (SOBE)  posed in  the quarter plane $\R^+ \times \R^+= (0, \infty)\times (0, \infty)$,
\begin{equation}\label{sobe}
\begin{cases}
u_{tt}-u_{xx}+\beta u_{xxxx}-u_{xxxxxx}+(u^2)_{xx}=0,\quad x,t\in \R^+,\\
u(x,0)=\varphi (x), u_t(x,0)=\psi ''(x),  \\
u(0,t)=h_1(t), u_{xx}(0,t)=h_2(t), u_{xxxx}(0,t)=h_3(t),
\end{cases}
\end{equation}
with $\b=\pm1 $ which  can serve as a model for waves generated by a wave maker at one end of  a channel, or for waves approaching shallow water from deep water.
The SOBE was first introduced by Christov, Maugin and Velarde \cite{28} with $\beta=-1$ and was initially derived from the Euler equation with shallow-wave assumption  intending to  amend  the ill-posed   modeling problem of the original Boussinesq equation
\begin{equation*}
u_{tt}-u_{xx}- u_{xxxx}+(u^2)_{xx}=0,
\end{equation*}
which had been used in a considerable range of applications such as coast and harbor engineering, simulation of tides and tsunamis. The SOBE was also proposed in modeling the nonlinear lattice dynamics in elastic crystals by Maugin \cite{66}. 

Plenty of  theoretical study has been placed on the SOBE not only due to its   reliable physical backgrounds but also for its structural similarities to both the KdV equation and the ``good" Boussinesq equation\footnote{We will refer the Boussinesq equation as the ``good" Boussiensq equation in the rest of the article.}(cf. \cite{96-2}).  Of all the mathematical studies on the SOBE, one of the most fundamental problems is its well-posedness issue which includes the initial value problem (IVP) and the initial boundary value problem (IBVP).

The theories on the IVP of the SOBE have been well-developed thanks to the tremendous achievements on the IVPs of dispersive equations, such as the KdV equation and the Bousssinesq equation in particular, during the past decades.  Esfahani, Farah and Wang \cite{35} first showed that the following system,
\begin{equation}\label{ivp}
\begin{cases}
u_{tt}-u_{xx}+(u^2)_{xx}+\beta u_{xxxx}-u_{xxxxxx}=0,\quad x\in \R,\\
u(x,0)=\varphi (x), u_t(x,0)= \psi''(x),
\end{cases}
\end{equation}
is locally analytically  well-posed in $H^s(\R)$  for initial data  $(\varphi,\psi)\in  H^s(\mathbb{R})\times H^{s-1}(\mathbb{R})$ with $s=0,1$, by applying the Strichartz type smoothing which inherited Linares' work \cite{59} on the Boussinesq equation.  Later, Esfahani and Farah \cite{34}  established the conclusion for  $s>-\frac{1}{2}$ by using a related Bourgain space and estimates inherited from Kenig, Ponce and Vega's work \cite{50} on the KdV equation. Finally, Esfahani and Wang \cite{37}  improved the result to $s>-\frac{3}{4}$  by using the $[k;Z]$-multiplier norm method introduced by Tao \cite{tao1}. 

 On the other hand, the less-developed theories on the IBVP of the dispersive equations raised  many open problems.
 Among different works on various dispersive equations,  the study on the KdV equation is relatively mature. Consider the KdV equation posed in a quarter plane,
 \begin{equation}\label{y-1}
 \begin{cases}
 u_t+u_x+u u_x+u_{xxx}=0, \quad (x,t)\in \R^+\times \R^+,\\
 u(x,0)=\phi(x),u(0,t)=h(t).\\
 \end{cases}
 \end{equation}
By applying the Laplace transform with respect to the temporal variable $t$ to find an explicit solution formula for the associated linear problem,
  \begin{equation} \label{y-2}
 \begin{cases}
 u_t+u_x+ u_{xxx}=0, \quad (x,t)\in \R^+\times \R^+,\\
 u(x,0)=0,u(0,t)=h(t),\\
 \end{cases}
 \end{equation}
Bona-Sun-Zhang   \cite{17} showed that  the IBVP (\ref{y-1}) is locally  analytically  well-posed in the space $H^s(\R^+)$ for $s>\frac34 $  with   $(\phi ,h)\in H^s(\R^+)\times H^{\frac{s+1}{3}}(\R^+) $  by adapting the approach developed by Kenig, Ponce and Vega  \cite{53} in dealing  with the IVP of the KdV equation.  During the same period,   Colliander-Kenig and Holmer   developed a different approach to handle the linear IBVP (\ref{y-2})  and showed that the IVP (\ref{y-1})  is  locally analytically well-posed   in the space $H^s(\R^+)$  for $s\geq 0$ \cite{31} and $s>-\frac34$ \cite{47} using a related Bourgain space, $X^{s,b}_{KdV}$, defined as 
  \[\|w\|_{X_{KdV}^{s,b}}=\left(\int_\R\int_\R (1+|\xi|)^{2s} (1+|\tau-(\xi^3-\xi)|)^{2b} |\hat{w}(\xi,\tau)|d\xi d\tau\right)^{\frac12}.\]
 Though both methods obtain the result for $s>-\frac34$, 
the result in \cite{15} is slightly stronger since it allows the critical value $b=\frac12$ in the corresponding $X^{s,b}$-space. This inspires our study on the improvement of the SOBE.
 
Both approaches  developed in \cite{15} and  \cite{47} have been applied to other dispersive equations. However, there are still many open problems left. For example, unlike the KdV equations, there are numerous dispersive equations, including  the nonlinear Schr\"odinger equation and the Boussinesq equation, sharp regularity results on  IBVPs when compared to their IVPs are not achieved (See e.g. \cite{1,89,cava,97,holmer,55}).  More precisely,  the Cauchy problem for the Boussinesq equation  can achieve its  analytically well-posedness in the space $H^s(\R)$ for $s>-\frac12$ \cite{55}, however, the best result on its IBVP  on the quarter plane  by far  is  $s>-\frac14$ \cite{97}.  Similar problem also appears in early study on the SOBE. For instance, it has been shown in \cite{96-1} that  the IBVP  of the SOBE posed in the quarter plane  is locally  analytically well-posed in the space $H^s (\R^+)$    for $s>-\frac12$ while the one for its IVP can be as low as $s>-\frac34$.

In review of the studies on the  IVP for dispersive equations under the $X^{s,b}$-spaces, the number $b$ is usually chosen to be strictly larger than $\frac12$ (See e.g. \cite{37,41,53,52,50,51,55}), while for the IBVPs, $b$ has to be strictly less than $\frac12$ (See e.g. \cite{15,14,cava,31,97,95,holmer,47,96-1}). Such a difference will sometimes  bring difficulties in the bilinear estimates and other related linear estimates,  which further prevents one obtaining sharp results as mentioned above. Moreover,  in the study of the  IVP for general dispersive equations, Tao \cite{tao1} introduced a systematical theory to deduct sharp estimate on the multi-linear estimate problem  on $X^{s,b}$ for $b>\frac12$.  Combining the ideas in \cite{15} and  \cite{tao1}, the result in \cite{96-2} allowed $b =\frac12$.  
In this article, we will continue the  study for the well-posedness of the SOBE and show that the IBVP is  locally analytically  well-posed in $H^s (\R^+)$  for $s>-\frac34$. Moreover, since the SOBE has some structural similarities (See e.g. \cite{96,96-1}) to   the Boussinesq equation and the nonlinear Schr\"odinger equation,  this may shed some light on the study for these two equations on related IBVPs.

 Let us first present the  local well-posedness result for system \eqref{sobe}. We  denote $\vec{h}:=(h_1,h_2,h_3)\in \H^s(\R^+)$ with $$\H^s(\R^+):=H^{\frac{s+1}{3}}(\R^+)\times H^{\frac{s-1}{3}}(\R^+)\times H^{\frac{s-3}{3}}(\R^+).$$ 
We also set $\langle x\rangle=\sqrt{1+x^2}$ and introduce  a related modified  Bougain-type space, $X_{SOBE}^{s,b}(\R^2)$,  for the sixth order Boussinesq equation as follows.

\begin{defi}\label{defi1}  ($X^{s,b,\si}$ space and its norm)
	\begin{itemize}
		
	\item[(i)] For $s,b\in\R $ and $ \sigma  \in \R^+$, define $X_{SOBE}^{s,b}$  and $\Lambda_\si$ to be the completions of the Schwartz class ${\mathcal S}(\R^2)$ with respect to the following norms
	\[\|w\|_{X_{SOBE}^{s,b}(\R^2)}=\left\|\langle\xi\rangle^s\langle|\tau|-\phi^{*}(\xi)\rangle^b\widehat{w}(\xi,\tau)\right\|_{L^2_{\xi,\tau}(\R^2)},\]
	\[\|w\|_{\Lambda_{\si}(\R^2)}=\left\|\chi_{|\xi|\leq 1}\langle\tau\rangle^\si \widehat{w}
(\xi,\tau)\right\|_{L^2_{\xi,\tau}(\R^2)},\]
	where $\phi^{*}(\xi)=\sqrt{\xi^6+\beta \xi^4+\xi^2}$ and $\widehat{{ w }}$ is  the Fourier transform on both  time and space of $w$.
	
	\item[(ii)]  For any $\Omega \subset \R^+$, 
	 \[  X_{SOBE}^{s,b}\left(\R^+\times \Omega\right):=X_{SOBE}^{s,b}|_{\R^+\times \Omega}\]
with the quotient norm,
\be\label{quotient norm}
\|u\|_{X_{SOBE}^{s,b}\left(\R^+\times \Omega\right)}:=\inf_{w\in X_{SOBE}^{s,b}(\R^2)}\{\|w\|_{X_{SOBE}^{s,b}(\R^2)}: w(x,t)=u(x,t)\ \mbox{on} \ \R^+\times \Omega\}.\ee

\item[(iii)]   $X^{s,b,\si}$ is the set of  all functions $w$ satisfying,
\[\|w\|_{X^{s,b,\si}(\R^2)}:=\|w\|_{X_{SOBE}^{s,b}(\R^2)}+\|w\|_{\Lambda_{\si}(\R^2)}< \infty\]
and for any $T>0$,
\[ X^{s,b ,\si}_T:=X^{s, b,\si}(\R^+\times (0,T) ).\]  
Moreover,
\[\X^{s,b,\si}_T:= C([0,T]; H^s(\R^+))\cap X^{s,b,\si} _T\]
with the norm
\[\|w\|_{\X^{s, b,\si}_T}=\l(\sup_{t\in (0,T)} \|w(\cdot,t)\|^2_{H^s(\R^+)}+\|w\|_{X^{s,b,\si}_T}\r)^{\frac12}.\]
	
	\end{itemize}
\end{defi}

\noindent With these notations, the local well-posedness result is stated as below.

\begin{thm}\label{conditional}
Let $-\frac34 <s\leq -\frac12$ and $r>0$  be given.  There exists $ \sigma_0=\sigma_0(s)\in(\frac12,1) $ such that for any $ \sigma\in(\frac12,\sigma_0] $, we can find  $T=T(r,s,\sigma) >0$   such that  if
 \[(\varphi,\psi,\vec{h})\in H^s(\R^+) \times H^{s-1}(\R^+)\times \H^s(\R^+)\]
 with
 $$ \ \ \|(\varphi, \psi)\|_{H^s(\R^+)\times H^{s-1}(\R^+)}+ \|\vec{h}\|_{\H^s(\R^+)}\leq r,$$
 then the IBVP \eqref{sobe} admits a unique solution 
 \[ u\in C([0,T];H^s (\R^+)) \]
 satisfying
 \begin{equation}
 \label{cond1} u\in \X_{T}^{s,\frac12,\si}
 \end{equation} 
 with
 \[ \| u\|_{\X_{T}^{s,\frac12,\si}  } \leq \alpha _{r,s,\sigma}  \left (\| (\varphi, \psi, \vec{h} )\|_{H^s(\R^+) \times H^{s-1}(\R^+)\times \H^s(\R^+)} \right )  \]
 where $\alpha _{r,s,\sigma}:\R^+\to \R^+ $ is a continuous function depending only  on $r$,  $s$ and $ \sigma $. Moreover  the corresponding solution map is real analytic.

 \end{thm}


To obtain this    well-posedness,  we apply  the same  approach as that used in the earlier work \cite{96-1}  but with two key necessary  modifications. 

First, we extend the boundary operator to the whole line differently. Same as the previous work, by using the Laplace transform to the IBVP for the linear equation,
\begin{equation}\label{01}
\begin{cases}
u_{tt}-u_{xx}+\beta u_{xxxx}-u_{xxxxxx}=0,\quad x>0,t>0,\\
u(x,0)=0, u_t(x,0)=0, \\
u(0,t)=h_1(t), u_{xx}(0,t)=h_2(t), u_{xxxx}(0,t)=h_3(t),
\end{cases}
\end{equation}
we obtain an explicit representation for the solution, and denoted the solution as a boundary operator in the form $u=W_{bdr}(\vec{h})$. 
We then establish similar estimates  as those in \cite{96-1}:
\begin{align*}
\sup_{t\geq 0} \|u(\cdot,t)\|_{H_x^s(\mathbb{R}^+)} + \sup_{x \geq 0}\|\partial&^j_xu(x,\cdot)\| _{H^{\frac{s-j+1}{3}}(\mathbb{R}^+)}+\|\eta(t)u(x,t)\|_{X_{SOBE}^{s,b}}\lesssim \|\vec{h}\|_{\H^s(\R^+)},
\end{align*}
with $j=0,2,4$ and $\eta$ being a cut-off function, but for $-\frac34<s\leq \frac12$ and $b=\frac12$. Since $u=W_{bdr}(\vec{h})$ is defined only on the quarter plane, $\R^+\times \R^+$,  we need to adopt a proper extension to evaluate it in the $X^{s,b}_{SOBE}(\R^2)$-space. In \cite{96-1}, we use a simple but effective ``zero extension"  inherited from  the IBVP for the Boussinesq equation \cite{97} to establish the desired estimate for $s>-\frac12$ with $b\leq \frac12$. However, estimates under such an extension do not hold for any dispersive equations with $s\leq -\frac12$  (See. \cite{96-2}). Therefore, we will need a more delicate extension to achieve estimates under a lower regularity condition. Thanks to the similarity in structure between the SOBE and the KdV equation, we will be able to establish desired estimates under a similar extension as in \cite{15}.

Secondly, we will establish some new bilinear estimates  for $s\leq-\frac12$ and $b=\frac12$.  In \cite{96-1}, we followed the method that came from the study on the KdV and Boussinesq equations \cite{31,97,47} by setting $b<\frac12$. However, regardless of the restriction on $s$ for the related linear estimates on the boundary operator, we are not able to achieve a proper bilinear estimate for $s\leq-\frac12$. The new extension allows us to consider the IBVP in a related Bourgain space $X^{s,b}$ with $b\leq\frac12$, and, as   mentioned earlier, we can deal with the case when $b\geq \frac12$. These two coincide on $b=\frac12$, which makes it possible to find a way to overcome the difficulties in the bilinear estimates in low regularity space.


The paper is organized as follows. In section 2, we will re-scale the system and introduce an alternative version of our main theorem. In addition, basic notations and lemmas, and existing results one the related IVP will also be introduced. In section 3, various linear estimates for different corresponding linear problems are discussed under proper spaces. The bilinear estimates for $s<-\frac12$ and $b=\frac12$ will be established  in Section 4.  The final section is used to  prove the local well-posedness  by contraction mapping principal. 

\section{Preliminary}

\setcounter{equation}{0}

In order to simplify the proof, we first re-scale the original IBVP problem:
\begin{equation}\label{fullpro}
\begin{cases}
u_{tt}-u_{xx}+\beta u_{xxxx}-u_{xxxxxx} +(u^2)_{xx}=0,\quad x>0\mbox{, }t>0,\\
u(x,0)=\varphi(x), u_t(x,0)=\psi''(x), \\
u(0,t)=h_1(t), u_{xx}(0,t)=h_2(t), u_{xxxx}(0,t)=h_3(t).
\end{cases}
\end{equation}
\noindent  by introducing  
\[u^\ld(x,t):= \ld^{-4} u(\ld^{-1} x,\ld^{-3} t), \quad x,t\in\R^+,\]
for $\ld\geq 1$. Then,  the IBVP (\ref{fullpro}) is  reduced to 
\begin{equation}\label{fullred}
\begin{cases}
u^\ld_{tt}-\a^2 u^\ld_{xx}+\a \beta u^\ld_{xxxx}-u^\ld_{xxxxxx} +((u^\ld)^2)_{xx}=0,\\
u^\ld(x,0)=\varphi^\ld(x), u^\ld_t(x,0)=(\psi^\ld)''(x), \\
u^\ld(0,t)=h^\ld_1(t), u^\ld_{xx}(0,t)=h^\ld_2(t), u^\ld_{xxxx}(0,t)=h^\ld_3(t),
\end{cases}
\end{equation}
with $\a=\ld^{-2}$,
\[\varphi^\ld(x)=\ld^{-4}\varphi(\ld^{-1}x), \quad  \psi^\ld(x)=\ld^{-5}\psi(\ld^{-1}x),\]
and
\[h_1^\ld(t)= \ld^{-4}h_1(\ld^{-3}t),\ \ h_2^\ld(t)= \ld^{-6}h_2(\ld^{-3}t),\ \ h_3^\ld(t)= \ld^{-8}h_3(\ld^{-3}t).\]
For $\ld \geq 1$ and $s\geq -1$, one has
\[\begin{cases}
\|(\varphi^\ld, \psi^\ld)\|_{H^s(\R^+)\times H^{s-1}(\R^+)}\leq \ld^{-\frac52}\|(\varphi , \psi ) \|_{H^s(\R^+)\times H^{s-1}(\R^+)},\\
\|(h_1^\ld, h_2^\ld,h^\ld_3)\|_{\H^s(\R^+)}\leq \ld^{-\frac52} \|(h_1, h_2,h_3)\|_{\H^s(\R^+)}.
\end{cases}\]
Hence, as $\ld\rightarrow \infty$, we have
\[\a\rightarrow 0, \ \ \|(\varphi^\ld, \psi^\ld)\|_{H^s(\R^+)\times H^{s-1}(\R^+)}\rightarrow 0, \ \ \|(h_1^\ld, h_2^\ld,h^\ld_3)\|_{\H^s(\R^+)}\rightarrow 0.\]
Thus, in order to prove local well-posedness of \eqref{fullpro}, it suffices to consider the case when $$\a,\quad\|(\varphi^\ld, \psi^\ld)\|_{H^s(\R^+)\times H^{s-1}(\R^+)},\quad \|(h_1^\ld, h_2^\ld,h^\ld_3)\|_{\H^s(\R^+)}$$ are sufficiently small. With such re-scaling, an alternative version of theorem for local well-posedness, which is also equivalent to Theorem \ref{conditional}, will be presented in Section 5. 

For the sake of simplicity in notations, we will neglect the $\lambda$ and denote $u^\lambda$ as $u$ in system \eqref{fullred}. Similar notations will also be applied for $\varphi$, $\psi$ and $h_j$ with $j=1,2,3$. The system \eqref{fullred} is  now rewritten as
\begin{equation}\label{hsb}
\begin{cases}
u_{tt}-\a^2 u_{xx}+\a \beta u_{xxxx}-u_{xxxxxx} +(u^2)_{xx}=0,\quad x>0\mbox{, }t>0,\\
u(x,0)=\varphi(x), u_t(x,0)=\psi''(x), \\
u(0,t)=h_1(t), u_{xx}(0,t)=h_2(t), u_{xxxx}(0,t)=h_3(t).
\end{cases}
\end{equation}
Then similar to the settings in Definition \ref{defi1}, the related Bourgain spaces with an additional parameter $\alpha$ are defined as below.
\begin{defi}( $X^{s,b,\si}_\alpha$ space and its norm)\label{Def, FR space}
	
	\begin{itemize}
		
		\item[(i)] For $s,b\in\R$, $X_\a^{s,b}$ and $\Lambda_\si$ denotes the completion of the Schwartz class ${\mathcal S}(\R^2)$ with
		\[\|w\|_{X_\a^{s,b}(\R^2)}=\left\|\langle\xi\rangle^s\langle|\tau|-\phi_{\a}
		(\xi)\rangle^b\widehat{w}(\xi,\tau)\right\|_{L^2_{\xi,\tau}(\R^2)},\]
		and
		\[\|w\|_{\Lambda_{\si}(\R^2)}=\left\|\chi_{|\xi|\leq 1}\langle\tau\rangle^\si \widehat{w}
		(\xi,\tau)\right\|_{L^2_{\xi,\tau}(\R^2)},\]
		where $\phi_{\a}(\xi)=\sqrt{\xi^6+\a \b \xi^4+\a^2\xi^2}$, with $0<\a\leq 1$. 
		\item[(ii)]  Let $X^{s,b,\si}_\alpha$ to be the set of all function $w$ satisfying,
		\[\|w\|_{X_\a^{s,b,\si}(\R^2)}:=\|w\|_{X_\a^{s,b}(\R^2)}+\|w\|_{\Lambda_{\si}(\R^2)}< \infty.\]
	
	\end{itemize}

\end{defi}

Again, for convenience of notations, we will drop the parameter $ \a $ in function $ \phi_{\a} $ which was defined above in part (i) of Definition \ref{Def, FR space}, that is, we will simply write $\phi_{\a}(\xi)$ as $ \phi(\xi) $ from now on.

For a better cancellation in  later study with Bourgain spaces, we introduce the following lemma.
\begin{lem}\label{equ}
	There exists a constant $c>0$ such that
	$$
	\frac{1}{c}\leq \sup_{x,y\geq 0} \frac{1+|x-\sqrt{y^3}-\frac{\a \beta}{2}\sqrt{y}|}{1+|x-\sqrt{y^3+\a \beta y^2+ \a^2 y}|}\leq c,\ \ \ \beta=\pm1.
	$$
\end{lem}
\noindent This  can lead to an equivalence between the norm of the Bourgain-type space for the sixth order Bousssinesq equation in \eqref{hsb} and the one for the KdV-type equation (See. \cite{54,53,52,50}), that is,
\[\|w\|_{X_\a^{s,b}(\R^2)}\sim \left\|\langle \xi\rangle^s\langle |\tau|-|\xi|^3-\frac{\a \beta}{2} |\xi|\rangle^b\widehat{w}(\xi,\tau)\right\|_{L^2_{\xi,\tau}(\R^2)}.\]

Next, the inequalities in Lemma \ref{lem1}--Lemma \ref{lem3} will be presented. These inequalities will be useful in the later study on the bilinear estimates and their proofs can be found in e.g. \cite{96-2, xin}.
\begin{lem}\label{lem1}
Let $\rho_{1}>1$ and $0\leq\rho_{2}\leq\rho_{1}$. Then there exists  a constant $C=C(\rho_{1},\rho_{2})$ such that for any $c_1,c_2\in\m{R}$,
\be\label{int in tau}
\int_{-\infty}^{\infty}\frac{dx}{\la x-c_1 \ra^{\rho_{1}} \la x-c_2 \ra^{\rho_{2}}}\leq \frac{C}{\la c_1-c_2\ra^{\rho_{2}}}.\ee
\end{lem}

\begin{lem}\label{lem2}
If $\rho>\frac{1}{2}$, then there exists $C=C(\rho)$ such that for any $c_{i}\in\m{R},\,0\leq i\leq 2$, with $c_{2}\neq 0$,
\be\label{bdd int for quad}
\int_{-\infty}^{\infty}\frac{dx}{\la c_{2}x^{2}+c_{1}x+c_{0}\ra^{\rho}}\leq \frac{C}{|c_{2}|^{1/2}}.\ee
Similarly, if $\rho>\frac{1}{3}$, then there exists $C=C(\rho)$ such that for any $c_{i}\in\m{R},\,0\leq i\leq 3$, with $c_{3}\neq 0$,
\be\label{bdd int for cubic}
\int_{-\infty}^{\infty}\frac{dx}{\la c_{3}x^{3}+c_{2}x^{2}+c_{1}x+c_{0}\ra^{\rho}}\leq \frac{C}{|c_{3}|^{1/3}}.\ee
\end{lem}

\begin{lem}\label{lem3}
Let $\rho>1$. Then there exists  a constant $C=C(\rho)$ such that  for any $c_i\in\m{R} $, $0\leq i\leq 2$, with $ c _2\ne 0$,
\be\label{int for quad}
\int_{-\infty}^{\infty}\frac{dx}{\la c_{2}x^{2}+c_{1}x+c_{0}\ra^{\rho}}\leq C\,|c_{2}|^{-\frac{1}{2}}\Big\la c_{0}-\frac{c_{1}^{2}}{4c_{2}}\Big\ra^{-\frac{1}{2}}.\ee
\end{lem}

Finally, some important notations and some of the previous results on the IVP  of a sixth order Boussinesq equation will be provided below. Denote  $u=[W_\R(f_1,f_2)](x,t)$ as the solution to the linear problem,
\begin{equation}\label{linear1}
\begin{cases}
u_{tt}-\a^2 u_{xx}+\a \beta u_{xxxx}-u_{xxxxxx}=0, \ \ \ x\in\R, t>0,\\
u(x,0)=f_1(x),u_t(x,0)=f_2''(x).
\end{cases}
\end{equation}
and write
$
    [W_{R}(f_1,f_2)](x,t):=[V_1(f_1)](x,t)+[V_2(f_2)](x,t)
$
with
\begin{equation}\label{v01}
   [ V_1(f_1)](x,t):=\frac{1}{2}\int_{\R} \left(e^{i(t\phi(\xi)+x\xi)}+e^{{i(-t\phi(\xi)+x\xi)}}\right)\widehat{f}_1(\xi)d\xi,
\end{equation}
and
\begin{equation}\label{v02}
   [ V_2(f_2)](x,t):=\frac{1}{2i}\int_{\R } \left(e^{i(t\phi(\xi)+x\xi)}-e^{i(-t\phi(\xi)+x\xi)}\right)\frac{\xi^2\widehat{f}_2(\xi)}{\phi(\xi)}d\xi.
\end{equation}
All through this article, we denote $\eta(t)$ to be a cut-off function such that $\eta\in C^\infty_0(\R)$ with $\eta(t)=1$ on $(-1,1)$ and supp $\eta\in (-2,2)$.  The estimates below  come from \cite{34,35,96}.
\begin{lem}\label{R}
For any $s,b\in\m{R}$, there exists some constant $ C=C(s,b) $ such that for any $f_1 \in H^s(\R)$ and $f_2\in H^{s-1}(\R)$, the solution $u$ of the IVP (\ref{linear1}) satisfies
\begin{align*}
\sup_{t\geq 0} \|\eta(t)u(\cdot,t)\|_{H^s(\R)}&\leq C \|f_1\|_{H^s(\R)}+\|f_2\|_{H^{s-1}(\R)},\\
 \sup_{x\geq 0} \|\eta(t) \partial ^j_x u(x,\cdot)\|_{H^{\frac{s-j+1}{3}}(\mathbb{R})}&\leq C\|f_1 \|_{H^s(\mathbb{R})}+\|f_2 \|_{H^{s-1}(\mathbb{R})},\\
   \|\eta(t)u(x,t)\|_{X_\a^{s,b}(\R^2)}&\leq C \|f_1\|_{H^s(\R)}+\|f_2\|_{H^{s-1}(\R)},
\end{align*}
for $j=0,1,2,..,5$.
\end{lem}
According to the Duhamel's principal, the  IVP for the forced linear equation,
 \begin{equation*}
\begin{cases}
u_{tt}-\a^2u_{xx}+\a \beta u_{xxxx}-u_{xxxxxx}=f_{xx}(x,t),\quad x\in \mathbb{R}\mbox{, }t>0,\\
u(x,0)=0, u_t(x,0)=0,
\end{cases}
\end{equation*}
has its solution $u$ in the form
\[u(x,t)=\int^t_0 [W_\R(0,f)](x,t-t')dt'.\]
The following conclusion of the above function $u$ is well-known, see e.g. Lemma 2.2 in \cite{34}, Lemma 2.2 in \cite{41} and Lemma 2.1 in \cite{GTV97}. 
\begin{lem}\label{for}
Let $-\frac12< b'\leq 0 \leq b\leq b'+1$ and $0<T\leq 1$.  Then for any $s\in\m{R}$ and $ 0<\a\leq 1 $, there exists a constant $C=C(b,b',s)$ such that
\begin{align*}
 \left\|\eta\Big(\frac{t}{T}\Big)\int^t_0 [W_\R(0,f)](x,t-t')dt'\right\|_{X_{\a}^{s,b}(\R^2)}\leq C\,T^{1+b'-b} \left\|\left(\frac{\xi^2\widehat{f}(\xi,\tau)}{\phi(\xi)}\right)^\vee\right\|_{X_{\a}^{s,b'}(\R^2)},
\end{align*}
where    `` $\vee$" denote the   inverse Fourier transform in both time and space.
\end{lem}

\section{Linear problems}
\setcounter{equation}{0}
This section is divided into two subsections. First, some explicit representation formulas for solutions of IBVPs for the sixth order Boussinesq equations are recalled from \cite{96}. The second subsection is devoted to show linear estimates which will play important roles in our later analysis.

\subsection{Linear Representations}
In this subsection, we present some existing results on the explicit formulas for different linear IBVPs related  to \eqref{fullred}.
First, let us consider the non-homogeneous boundary value problem with zero initial conditions
\begin{equation}\label{homoin}
\begin{cases}
u_{tt}-\a^2 u_{xx}+\a \beta u_{xxxx}-u_{xxxxxx} =0,\quad x>0\mbox{, }t>0,\\
u(x,0)=0, u_t(x,0)=0, \\
u(0,t)=h_1(t), u_{xx}(0,t)=h_2(t), u_{xxxx}(0,t)=h_3(t).
\end{cases}
\end{equation}
We only study the case for $\b=1$ since the case for $\b=-1$ is similar. We denote $\gamma _j (\rho)$ for $\ j=1,2,3$, to be  the three solutions  of the characteristic equation
\begin{equation}\label{ch}
    \gamma^6-\a\gamma^4+\a^2\gamma^2-\rho^2=0,
\end{equation}
with $Re ( \gamma _j(\rho)) <0$  and distinct.  We also denote
 $c_j(\rho), \mbox{ for }j=1,2,3$,  to be the solutions to the linear system
\begin{equation}\label{system}
\begin{cases}
c_{1}+c_{2}+c_{3}=\widetilde{h}_1(\rho),\\
\gamma_{1}^2c_{1}+\gamma_{2}^2c_{2}+\gamma_{3}^2c_{3}=\widetilde{h}_2(\rho),\\
\gamma_{1}^4c_{1}+\gamma_{2}^4c_{2}+\gamma_{3}^4c_{3}=\widetilde{h}_3(\rho),
\end{cases}
\end{equation}
with
\begin{equation*}
   \widetilde{h}_j(\rho)=\int^{+\infty}_0 e^{-\rho t}h_j(t)dt, \hspace{0.5cm}j=1,2,3.
\end{equation*}
For $ j=1,2,3 $, $ c_{j}(\rho) $ can be computed based on the Cramer's rule. Let $\Delta(\rho)$ be the determinant of the coefficient matrix of \eqref{system} and $\Delta_{j}(\rho)$ be the determinants of the matrices with the $j$-th column replaced by $(\widetilde{h}_1(\rho),\widetilde{h}_2(\rho),\widetilde{h}_3(\rho))^T$ for $j=1,2,3$. For $j,m=1,2,3$,  $\Delta_{j,m}(\rho)$ is  obtained from $\Delta_j(\rho)$ by letting $\widetilde{h}_m(\rho)=1$ and $\widetilde{h}_j(\rho)=0$ for $j\neq m$.
Then the  solution of \eqref{homoin} can be written in the form
\begin{align}
u(x,t)=[W_{bdr}(h_1,h_2,h_3)](x,t)
=\sum^3_{m=1}\l(\emph{\textrm{I}}_m (x,t)+\sum^3_{l=1} W^l_{bdr,m}(h_m)(x,t)+ \overline{ I}_m (x,t)\r),\label{Wbdr}
\end{align}
where
\begin{equation}\label{W10}
 W^1_{bdr,m}(h_m)(x,t): = \sum^3_{j=1}\int^{i}_{1+i}e^{\rho t}\frac{\Delta_{j,m}(\rho)}{\Delta(\rho)}
e^{\gamma_{j}(\rho)x}\hat{h}_m(\rho)d\rho,
\end{equation}
\begin{equation}\label{W20}
 W^2_{bdr,m}(h_m)(x,t): = \sum^3_{j=1}\int^{1+i}_{1-i}e^{\rho t}\frac{\Delta_{j,m}(\rho)}{\Delta(\rho)
}e^{\gamma_{j}(\rho)x}\hat{h}_m(\rho)d\rho,
\end{equation}
\begin{equation}\label{W30}
 W^3_{bdr,m}(h_m)(x,t): = \sum^3_{j=1}\int^{1-i}_{-i}e^{\rho t}\frac{\Delta_{j,m}(\rho)}{\Delta(\rho)}
e^{\gamma_{j}(\rho)x}\hat{h}_m(\rho)d\rho,
\end{equation}
\begin{equation}\label{W40}
    I_m(x,t)=\sum^3_{j=1}\int^{+\infty}_1 e^{i\phi(\mu)t}e^{\gamma_{j}^+(\mu)x}
    \frac{\Delta_{j,m}^+(\mu)}{\Delta^+(\mu)}\frac{3\mu^4+2\a\mu^2+\a^2}{\sqrt{\mu^4+\a\mu^2+\a^2}}\widetilde{h}_m^+(\mu)d\mu,
\end{equation}
with
\[\widetilde{h}_m^+(\mu)=\int^\infty_0 e^{-i\phi(\mu )s}h_m(s)ds, \mbox{ for } \phi(\mu)=\sqrt{\mu^6+\a\mu^4+\a^2\mu^2}.\]
and
\begin{equation}\label{gamma}
  \gamma_1^+ (\mu) = -i\mu, \quad  \gamma_2^+ (\mu) = -p(\mu)-iq(\mu), \quad  \gamma_3^+ (\mu) = -p(\mu)+iq(\mu),
\end{equation}
with
\[p(\mu)=\frac{1}{\sqrt{2}}\left(\sqrt{\mu^2+\a^2+\sqrt{4\mu^4+4\a\mu^2+4\a^2}}\right),\]
\[q(\mu)=\frac{1}{\sqrt{2}}\left(\sqrt{\sqrt{4\mu^4+4\a\mu^2+4\a^2}-\mu^2-\a^2}\right).\]
\noindent For more details on the derivation of the representation, the reader may refer to \cite{96-1,96-2}.

Next, we move on to the linear problem with non-zero initial condition but with zero boundary values,
\begin{equation}\label{in2-0}
    \begin{cases}
    u_{tt}-\a^2 u_{xx}+ \a \beta u_{xxxx}-u_{xxxxxx}=0,\quad x>0 \mbox{, }t>0,\\
    u(x,0)=\varphi (x), u_t(x,0)=\psi ''(x),\\
    u(0,t)=0, u_ {xx}(0,t)=0, u_{xxxx}(0,t)=0.
    \end{cases}
\end{equation}
Let   $\varphi ^*, \ \psi ^*$   be  zero extensions of $\varphi $ and $\psi $ from $\R^+\to \R$,  respectively.
We denote \[p(x,t)=[ W_{R} (\varphi ^*, \psi ^*)] (x,t),\]   and set $\vec{p}=(p_1,p_2,p_3)$ with 
\begin{equation}\label{p0}
 p_1 (t)=p(0,t), \quad p_2 (t)= p_{xx} (0,t), \quad p_3 (t)=p_{xxxx} (0,t).
\end{equation}
Then the solution $u $ of the IBVP (\ref{in2-0}) can be written in the form
\[ u(x,t) = [W_{R} (\varphi ^*, \psi ^*)] (x,t)- [W_{bdr} (\vec{p})] (x,t) .\]

Finally, we consider the forced linear problem with zero boundary and zero initial conditions,
\begin{equation}\label{in2-1}
    \begin{cases}
    u_{tt}-\a^2 u_{xx}+ \a \beta u_{xxxx}-u_{xxxxxx}=f_{xx}(x,t),\quad x>0 \mbox{, }t>0,\\
    u(x,0)=0, u_t(x,0)=0,\\
    u(0,t)=0, u_ {xx}(0,t)=0, u_{xxxx}(0,t)=0.
    \end{cases}
\end{equation}
Let $f^* $ be an  extension of $f$ from $\R^+ $ to $\R$. Then, the corresponding solution $u$ of (\ref{in2-1}) can be written as
\[ u(x,t) = \int ^t_0 [  W_R (0,f^*)](x, t-t' ) dt' - [W_{bdr} (\vec{q})](x,t),\]
where $\vec{q}=(q_1,q_2,q_3)$ and
\begin{equation}\label{q0}
q_1 (t)=q(0,t), \quad q_2 (t)= q_{xx} (0,t), \quad q_3 (t)=q_{xxxx} (0,t)
\end{equation}
with
\[ q(x,t) = \int ^t_0 [  W_R (0,f^*)](x, t-t' ) dt' .\]

We end up this subsection by combining all the above conclusions to address the fully nonhomogeneous forced linear IBVP,
\begin{equation}\label{full}
    \begin{cases}
    u_{tt}-\a^2 u_{xx}+ \a \beta u_{xxxx}-u_{xxxxxx}=f_{xx}(x,t),\quad x>0 \mbox{, }t>0,\\
    u(x,0)=\varphi(x), u_t(x,0)=\psi''(x),\\
    u(0,t)=h_1(t), u_ {xx}(0,t)=h_2(t), u_{xxxx}(0,t)=h_3(t).
    \end{cases}
\end{equation}
Since we only consider solutions in  Sobolev spaces $H^s(\R^+)$ wiht $s<  \frac12$, then there is no need to handle the compatibility conditions between the initial and boundary conditions. Hence, we have the following proposition.
\begin{pro}\label{represent}
The solution $u(x,t)$ of \eqref{full} can be written in the form
\begin{align*}
u(x,t)= W_R(\varphi^*,\psi^*)(x,t)+\int^t_0 W_R(0,f^*)(x,t-\tau)d\tau
+W_{bdr}(\vec{h}-\vec{p}-\vec{q}),
\end{align*}
where $\vec{h}=(h_1,h_2,h_3)$, and $\vec{p}$, $\vec{q}$ are defined in \eqref{p0}, \eqref{q0} respectively.
\end{pro}
\subsection{Linear Estimates}
  In this subsection, we present  estimates  for the linear operators shown in Proposition \ref{represent}, which play an essential role to establish the well-posedness for our problem.
\begin{lem}\label{Lemma, lin est}
Let $s\in\m{R}$, $\frac12<\si< 1$, $ 0<\a \leq 1 $ and $0<T\leq 1$. Then there exists a constant $C=C(s,\sigma)$ such that
\be\label{R1}
\|\eta(t) W_R(f_1,f_2)\|_{X^{s,\frac12,\si}}\leq C\big( \|f_1\|_{H^s(\R)}
+\|f_2\|_{H^{s-1}(\R)}\big), 
\ee
\be\label{R2}
\l\|\eta\Big(\frac{t}{T}\Big) \int_{0}^{t} W_R(0,f)(x,t-t')dt'\r\|_{X_\a^{s,\frac12,\si}}\leq C \l\|\l(\frac{\xi^2 \widehat{f}}{\phi(\xi)}\r)^\vee\r\|_{X_\a^{s,\si-1}}.
\ee
\end{lem}
\noindent The proof for this lemma follows directly from Lemma \ref{R} and \ref{for} by noticing
\[X_\a^{s,b} \subset X_\a^{s,\frac12,\si}, \quad \mbox{for all } b>\si>\frac12. \]

\begin{lem}\label{lemma3}
For $-\frac34 < s \leq -\frac12$, there exist $ C=C(s) $ and an extension  $\Phi_{bdr}(\vec{h})$ with
\[\Phi_{bdr}(\vec{h})= W_{bdr}(\vec{h}), \quad \mbox{ for }\quad  (x,t)\in \R^+\times [0,1],\]
such that for any $\si\in(\frac12,1)$, one has 
\be\label{bd2}
\sup_{x\geq0}\|\partial_x^j[ \eta(t)\Phi_{bdr}(\vec{h})](x,\cdot)\|_{H^{\frac{s-j+1}{3}}(\R^+)}\leq C \|\vec{h}\|_{\H^s(\R^+)},\ee
\be\label{bd1}\sup_{t\geq0}\|[\eta(t)\Phi_{bdr}(\vec{h})](\cdot,t)\|_{H^{s}(\R^+)}\leq C \|\vec{h}\|_{\H^s(\R^+)},\ee
\be\label{bd}
\|\eta(t) \Phi_{bdr}(\vec{h})\|_{X_\a^{s,1/2,\si}} \leq C \|\vec{h}\|_{\H^s(\R^+)}.
\ee
where $j=0,1,2...,5.$
\end{lem}

\begin{proof}
	The proof of estimate \eqref{bd2} can be referred to Proposition 3.5 in  \cite{96-1}. 
	For estimates \eqref{bd1} and \eqref{bd},  it suffices to carry out the proof for  $\vec{h}=(h_1,0,0)$. Let us recall  \eqref{Wbdr} and write 
	\[W_{bdr}(\vec{h})=I_1+\bar{I}_1+\sum_{l=1}^3 W_{bdr,1}^l(h_1),\]
	as defined in \eqref{W10}-\eqref{W40}. Although the value of $s$ and $b$ in   Proposition 3.4 and 3.5 in \cite{96-1}  are different from here,  the proof of estimates for  $W^l_{bdr,1}(h_1)$ with $l=1,2,3$ is  still valid  in this lemma. 	It thus suffices to establish the conclusion for $I_1$ and $\bar{I}_1$. Recall from  \eqref{W40}-\eqref{gamma},  we write $I_1$ as 
	\begin{align*}
	I_1(x,t)&=\sum^3_{j=1}\int^{+\infty}_1 e^{\gamma_{j}^+(\mu)x}e^{i\phi(\mu)t}
	\frac{\Delta_{j,m}^+(\mu)}{\Delta^+(\mu)}\frac{3\mu^4+2\a\mu^2+\a^2}{\sqrt{\mu^4+\a\mu^2+\a^2}}\int^\infty_0 e^{-i\phi(\mu )s}h_1(s)dsd\mu\\
	:&=\sum^3_{j=1}I_{1,j}(x,t).	
	\end{align*}
	Again, the estimates for $I_{1,1}$ with $\gamma_1^+(\mu)=-i\mu$ can be obtained as in \cite{96-1} regardless  the restrictions on $s$ and $b$. Finally, we will establish the estimates for $I_{1,2}$ and $\bar{I}_{1,2}$  by adapting the extension inherited from Bona-Sun-Zhang \cite{15},   estimates for $I_{1,3}$ and $\bar{I}_{1,3}$ can be achieved similarly.  Let us denote $\mu(\lambda)$ be  the  solution of  $$\lambda=\phi(\mu)=\sqrt{\mu^6+\a \mu^4+\a^2\mu^2}$$ for $\mu\geq 1$ and $\lambda\geq a$ with $a=\sqrt{1+\a+\a^2}$, which leads to 
	\begin{align}
	I_{1,2}+\bar{I}_{1,2}&=2 Re \int^{+\infty}_a \int^\infty_0 e^{\left(-p(\mu(\lambda))-iq(\mu(\lambda))\right)x} e^{i\phi(\mu(\lambda))(t-s)}
	\frac{\Delta_{2,1}^+(\mu(\lambda))}{\Delta^+(\mu(\lambda))}h_1(s)ds d\lambda\nonumber\\
	 &= 2 \int^{+\infty}_a \int^\infty_0 e^{-p(\mu(\lambda))x} \cos{\Big(\phi(\mu(\lambda))(t-s)-q(\mu(\lambda))\Big)}
	 \frac{\Delta_{2,1}^+(\mu(\lambda))}{\Delta^+(\mu(\lambda))}h_1(s)ds d\lambda\nonumber\\
	 :&= E(x,t)\label{Ee}
	 \end{align}
	for $x\geq 0$. We now extend $E(x,t)$ from $\R^+$ to $\R$ by defining
	\[E^*(x,t)=\begin{cases}
	E(x,t), \quad x\geq 0\\
	g(x,t), \quad x<0,
	\end{cases}\]
	where $g(x,t)$ is defined exactly as formula (2.14) in Bona-Sun-Zhang \cite{15} for the KdV equation. Recall that the Bourgain space for the sixth order Boussinesq equation,
	\[\|w\|_{X_\a^{s,b}(\R^2)}\sim \|\langle \xi\rangle^s\langle |\tau|-|\xi|^3-\frac{\a }{2} |\xi|\rangle^b\widehat{w}(\xi,\tau)\|_{L^2_{\xi,\tau}(\R^2)}.\]
	 is very similar to the one for the KdV equation,
	 	\[\|w\|^{KdV}_{X^{s,b}(\R^2)}=\|\langle \xi\rangle^s\langle \tau-\xi^3+ \xi\rangle^b\widehat{w}(\xi,\tau)\|_{L^2_{\xi,\tau}(\R^2)}.\]
	  We can then apply Bona-Sun-Zhang's strategy for the desired estimates on $E^*(x,t)$ by noticing the following two facts. First, the proof for the KdV equation in \cite{15} does not take the sign for $\tau$ and $\xi$ into account. Secondly, for $\a$ small, one has 
	  \[\phi(\mu)=\sqrt{\mu^6+\a \mu^4+\a^2\mu^2}\sim \mu^3, \quad p(\mu ) \sim \mu,\quad q(\mu)\sim \mu,\]
	  for $\mu\geq 1$, thus the influence of $\a$ and slightly differences in $E(x,t)$ (See. \eqref{Ee}) can be ignored. 
\end{proof}

Next, we provide an alternative version of Kato smoothing for the forcing term with slightly higher regularities as needed  in Lemma \ref{Lemma, lin est} as well as in Proposition 3.7 in \cite{96-1}. This is because that we can not fully convert the regularities from time to space within optimal conditions when including derivatives on space in \eqref{f2}. One way to overcome this problem is to adapt the idea from the KdV equation \cite{47}, Coupled-KdV system \cite{96-2} and   Boussinesq-type equations  \cite{97,96}  by compensating another modified norm on the right hand side of \eqref{f2}. Other than that, in this article, we do not adding the modified norm but relax the regularity restriction. Due to the specialty of the  nonlinear term for the SBOE, we can still achieve a proper bilinear estimate under this relaxed norm in the next section.

\begin{pro}\label{kato}
For $-1\leq s\leq 0$ and $\si>\frac12$, there exists $ C=C(s,\si) $ such that for any $j=1,2...,5$,
\begin{equation}\label{f2}
\left\|\eta(t)\int^t_0 \partial_x^j[ W_R(t-t')(0 ,f)](x,t')dt' \right\|_{H_t^{\frac{s-j+1}{3}}(\R)}\leq C \l\|\l(\frac{\xi^2 \widehat{f}}{\phi(\xi)}\r)^\vee\r\|_{X_\a^{s+2,\si-1}}.
\end{equation}
\end{pro}

\begin{proof}
The following proof follows from the idea in \cite{15} and \cite{53}. We write
\begin{align*}
    \int^t_0  [W_R(t-t')]&(0, f(x,t'))dt' = \int_0^t \int_\R  e^{ix \xi }\l(e^{i\phi(\xi)(t-t')}-e^{-i\phi(\xi)(t-t')}\r)\frac{\xi^2\widehat{f}^x(\xi,t')}{i\phi(\xi)}d\xi dt',
\end{align*}
where $\widehat{f}^x$ denote the Fourier transform with respect to  $x$.
It then suffices to establish the estimate for
\begin{align*}
A:=&\int_0^t \int_\R e^{ix \xi }e^{i\phi(\xi)(t-t')}\frac{\xi^2\widehat{f}^x(\xi,t')}{i\phi(\xi)}d\xi dt'\\
=& \int_\R \int_\R e^{ix \xi }\frac{\xi^2\widehat{f}(\xi,\tau)}{i\phi(\xi)}e^{i\phi(\xi)t}\int^t_0 e^{it'(\tau-\phi(\xi))}dt'd\xi d\tau\\
=& -\int_\R \int_\R e^{ix \xi }\frac{\xi^2\widehat{f}(\xi,\tau)}{\phi(\xi)} \frac{e^{it\tau}-e^{it\phi(\xi)}}{\tau-\phi(\xi)}d\xi d\tau\\
:=&-(A_1+A_2-A_3),
\end{align*}
where
\begin{align*}
A_1&=\int_{\R^2}e^{ix\xi}\frac{e^{it\tau}-e^{it\phi(\xi)}}{\tau-\phi(\xi)}\theta
(\tau-\phi(\xi))\frac{\xi^2
	\widehat{f}(\xi,\tau)}{\phi(\xi)}d\xi d\tau ,\\
A_2&=\int_{\R^2}e^{ix\xi}\frac{e^{it\tau}}{\tau-\phi(\xi)}\theta^c(\tau-\phi(\xi))
\frac{\xi^2\widehat{f}(\xi,\tau)}{\phi(\xi)}d\xi d\tau ,\\
A_3&=\int_{\R^2}e^{ix\xi}\frac{e^{it\phi(\xi)}}{\tau-\phi(\xi)}\theta^c(\tau-\phi(\xi))
\frac{\xi^2\widehat{f}(\xi,\tau)}{\phi(\xi)}d\xi d\tau ,
\end{align*}
with
$\theta$ being a cut-off function such that $\theta=1$ on $[-1,1]$ and $\text{supp} (\theta )\subset (-2,2)$, and  denoting $\theta^c=1-\theta$.  

For $A_1$, using Taylor expansion, one has
\begin{equation*}
A_1=\int_{\R^2}e^{ix\xi}e^{it\phi(\xi) }\sum^{\infty}_{k=0}\frac{(it)^{k}(\tau-\phi(\xi))^{k-1}}{k!}\theta(\tau-\phi(\xi))
\frac{\xi^2\widehat{f}(\xi,\tau)}{\phi(\xi)}d\tau d\xi.
\end{equation*}
Besides,  we can re-write $A_1$ as
\[A_1=\sum^{\infty}_{k=0}\frac{(it)^k}{k!} W^*_R( \Psi, 0),\]
where
\[\widehat{\Psi}=\int_\R(\tau-\phi(\xi))^{k-1}\theta(\tau-\phi(\xi))
\frac{\xi^2\widehat{f}(\xi,\tau)}{\phi(\xi)}d\tau,\] 
here, the operator $W^*_R $ is defined similar to \eqref{v01} but with the first part.
Hence, according to Lemma \ref{R} and setting $F(\xi,\tau)=\frac{\xi^2\widehat{f}(\xi,\tau)}{\phi(\xi)}$, one has,
\begin{align*}
&\|\eta(t)\partial_x^j A_1\|_{H_t^{\frac{s-j+1}{3}}(\R)}\lesssim  \sum^\infty_{k=0}\frac{\|t^k \eta(t)\|_{H^1(\R)}}{k!} \|\Psi\|_{H^s_x(\R)}\\
\lesssim& \sum^\infty_{k=0}\frac{\|t^k \eta(t)\|_{H^1(\R)}}{k!}\left\|\int_\R \la \xi \ra^s (\tau-\phi(\xi))^{k-1}\theta(\tau-\phi(\xi))F(\xi,\tau) d\tau \right\|_{L_x^{2}(\R)}\\
\lesssim & \l\|\int_{|\tau-\phi(\xi)|\leq1} \frac{1}{\la \xi \ra^{ 2}\la |\tau|-\phi(\xi)\ra^{ \si}} \frac{\la \xi \ra^{s+2}F(\xi,\tau)}{\la |\tau|-\phi(\xi)\ra^{1-\si}}d\tau\r\|_{L^2_x(\R)}\\
\lesssim &  \l(\int_\R \frac{1}{\la \xi \ra^{ 4}\la |\tau|-\phi(\xi)\ra^{ 2\si}} d\tau \r)^{\frac12}\|F\|_{X_\a^{s+2,\si-1}}:=G_1(\xi)\|F\|_{X_\a^{s+2,\si-1}},
\end{align*}
by noticing that $|\tau-\phi(\xi)|\geq ||\tau|-\phi(\xi)|$.
Similarly, for $A_3$, we have the re-write as,
\[ A_3=   W_R^*(\Phi,0),\]
where
\[\widehat{\Phi}(\xi)=\int_\R \frac{\th^c(\tau-\phi(\xi))F(\xi,\tau) }{\tau-\phi(\xi)}d\tau.\]
Then, according to Lemma \ref{R}, one has
\begin{align*}
\|\eta \partial_x^j A_3\|_{H_t^{\frac{s+j-1}{3}}(\R)}\lesssim \|\Phi\|_{H^{s}(\R)}\lesssim &\left\|\la\xi\ra^{s}\int_\R\frac{\th^c(\tau-\phi(\xi))F(\xi,\tau) }{\tau-\phi(\xi)}d\tau \right\|_{L^2_\xi}\\
\lesssim & \l\|\int_\R \frac{1}{\la\xi\ra^{ 2}\la|\tau|-\phi(\xi)\ra^{\si}}\frac{\la\xi\ra^{s+2}F}
{\la|\tau|-\phi(\xi)\ra^{1-\si}}\r\|_{L^2_\xi}\\
\lesssim & G_1(\xi) \|F\|_{X_\a^{s+2,\si-1}}.
\end{align*}
Since, for $\si>\frac12$,
\[G_1^2(\xi)=\frac{1}{\la \xi \ra^4} \int_\R \frac{1}{\la|\tau|-\phi(\xi)\ra^{2\si}}d\tau\lesssim \frac{1}{\la \xi \ra^4}\lesssim 1,\]
 we  can complete the proof for $A_1$ and $A_3$. While for $A_2$, we first show the estimate for $j=0,1,2$. It suffices to proof the case for $j=2$. According to Plancherel theorem, it follows that,
\begin{align*}
\|\eta(t)  \partial_x^2 A_2\|_{H_t^{\frac{s-1}{3}}(\R)}\lesssim & \left\|\int_{\R^2} \la\xi\ra^2 e^{it\tau} e^{ix\xi} \frac{\theta^c(\tau-\phi(\xi))}{\tau-\phi(\xi)} F(\xi,\tau)d\xi d\tau\right\|_{H_t^{\frac{s-1}{3}}(\R)}\nonumber\\
\lesssim &\left\|\langle\tau\rangle^{\frac{s-1}{3}} \int_\R\frac{\theta^c(\tau-\phi(\xi))}{\tau-\phi(\xi)}\la\xi\ra^2 F(\xi,\tau)d\xi\right\|_{L^2_\tau}\nonumber\\
\lesssim & \left\|\langle\tau\rangle^{\frac{s-1}{3}} \int_\R\frac{\la\xi\ra^2 F(\xi,\tau)}{\langle\tau-\phi(\xi)\rangle}d\xi\right\|_{L^2_\tau}\label{A2}\\
\lesssim & \bigg[\int_\R \langle\tau\rangle^{\frac{2s-2}{3}} G_2(\tau)
\bigg(\int_\R\frac{\la\xi\ra^{2s+4}
F^2(\xi,\tau)}{\la |\tau|-\phi(\xi)\ra^{2-2\si}}d\xi\bigg) d\tau\bigg]^\frac{1}{2},
\end{align*}
with
\begin{align*}
G_2(\tau)&= \int_\R \frac{d\xi}{\la \xi\ra^{2s} \la |\tau|-\phi(\xi)\ra^{2\si}}.
\end{align*}
We can readily check that  
\[\langle\tau\rangle^{\frac{2s-2}{3}}G_2(\tau)  \lesssim 1,\]
for $|\xi|\leq 1$ and $s\leq0$. While for $|\xi| >1$, by changing variable with $z=\phi(\xi)$, one has.
\begin{align*}
G_2(\tau)\lesssim \int_\R \frac{dz}{ |z|^{\frac23}\la z\ra^{\frac{2s}{3}} \la |\tau|-z\ra^{2\si}}.
\end{align*}
For $-1\leq s\leq 0$ and $\si>\frac12$, we can estimate the above integral for $|z|>1$, and it leads to 
\[G_2(\tau)\lesssim \int_\R \frac{dz}{\la |\tau|-z\ra^{2\si}} =\int_\R \frac{dz}{\la z\ra^{2\si}}\lesssim 1.\]
  Hence, we complete the proof for $A_2$ with $j=2$. Next, we consider the case for $j=3$.
 It has shown (c.f.  Lemma 2.12 in \cite{96}) that
\begin{equation*}
 \partial_x^3\left(\int^t_0[W_R(t-t')](0,f(x,t'))dt'\right)=\partial_t\left(\int^t_0[W_R(t-t')](0,f(x,t'))dt'\right).
\end{equation*}
Therefore, it yields that,
\begin{align*}
  &\left\| \eta(t)\int^t_0 \partial_x^3 [W_{R}(t-t')](0,f(x,t'))dt'\right\|_{H_t^{\frac{s-2}{3}}(\R)}\\
  = &\left\| \eta(t)\partial_t\left(\int^t_0  [W_{R}(t-t')](0,f(x,t'))dt'\right)\right\|_{H_t^{\frac{s-2}{3}}(\R)}\\
  \lesssim & \left\| \eta(t)\int^t_0  [W_{R}(t-t')](0,f(x,t'))dt'\right\|_{H_t^{\frac{s+1}{3}}(\R)},
\end{align*}
 and this reduces the proof to the case $j=0$. Similar arguments hold for $j=4,5$. The proof is now complete.
\end{proof}
\section{Bilinear estimate}
\setcounter{equation}{0}
In this section, we will handle the bilinear estimate. Thanks to the nonlinear structure of the SOBE, we can utilize more smoothing for the nonlinear term as needed comparing to its IVP.  As we mentioned in the previous section, we have relaxed some regularities in the estimate for the forcing term, and the bilinear estimate that we presented below still works for that condition!
\begin{thm}\label{bilinear}
For any $ s> -\frac34$, there exists $\si_0=\si_0(s)>\frac12$ such that for any $\si\in (\frac12, \si_0]$ and $ \a\in(0,1] $, one has the following bilinear estimate,
\be\label{bilin}
\l\|\l(\frac{\xi^2\widehat{uv}}{\phi(\xi)}\r)^\vee\r\|_{X_\a^{s+2,\si-1}}\leq C \|u\|_{X_\a^{s,\frac12,\si}}\|v\|_{X_\a^{s,\frac12,\si}}, \ee
where $C=C(s,\sigma)$.
\end{thm}

\begin{proof}
The idea of the proof is inherited from \cite{tao2,96-2,tao1,xin}.  By duality and the Plancherel's identity, it suffices to prove 
\be\label{bilin, duality}
\iint_{\m{R}^2} \frac{\xi^2\widehat{uv}}{\phi(\xi)}\,\wh{w}\,d\xi\,d\tau \leq C \|u\|_{X_\a^{s,\frac12,\si}}\|v\|_{X_\a^{s,\frac12,\si}}\|w\|_{X_\a^{-s-2,1-\si}},\ee
for any $u, v\in X_{\a}^{s,\frac12,\si}$ and $w\in X_{\a}^{-s-2,1-\si}$.   Next, we will reduce (\ref{bilin, duality}) to an estimate of some weighted convolution of $L^{2}$ functions.  

Denote $L=|\tau|-|\xi|^3-\frac{\a\b}{2}|\xi|$ and write
\be\label{f, bilin-l2 est}
\left\{\begin{array}{l}
f_1(\xi,\tau) =\la \xi\ra^s \l( \la L \ra ^{\frac12}+\chi_{|\xi|\leq 1}\la L \ra^{\si}\r) \hat{u}(\xi,\tau),  \vspace{0.05in}\\
 f_2(\xi,\tau) =\la \xi\ra^s \l( \la L \ra ^{\frac12}+\chi_{|\xi|\leq 1}\la L \ra^{\si}\r)  \hat{v}(\xi,\tau),  \vspace{0.05in}\\
f_3(-\xi,-\tau) =\la \xi\ra^{-s-2} \la L \ra^{1-\si} \hat{w}(\xi,\tau).  
\end{array}\right.\ee


\noindent For the convenience of notations, we denote $\rho=-s$, 
\be\label{int domain}
A:=\Big\{(\vec{\xi},\vec{\tau})\in\m{R}^{6}:\sum_{i=1}^{3}\xi_{i}=\sum_{i=1}^{3}\tau_{i}=0\Big\},\ee
where $\vec{\xi}=(\xi_{1},\xi_{2},\xi_{3})$ and $\vec{\tau}=(\tau_{1},\tau_{2},\tau_{3})$.  Then one can re-write (\ref{bilin, duality}) to
\be\label{bi1}
\int_A \frac{\xi_3^2\la\xi_1\ra^{\rho}\la\xi_2\ra^{\rho}|f_1(\xi_1,\tau_1)f_2(\xi_2,\tau_2)f_3(\xi_3,\tau_3)|}{\phi(\xi_3) \la\xi_3\ra^{\rho-2}  M_1 M_2\la L_3\ra^{1-\si}}\lesssim \prod^3_{j=1}\|f_j\|_{L^2_{\xi,\tau}},
\ee
where
\[M_1:= \la L_1 \ra ^{\frac12}+\chi_{|\xi_1|\leq 1}\la L_1 \ra^{\si},\quad  M_2:= \la L_2 \ra ^{\frac12}+\chi_{|\xi_2|\leq 1}\la L_2 \ra^{\si},\]
with
\begin{equation}\label{l1}
L_1:=|\tau_1|-|\xi_1|^3-\frac{\a\b}{2}|\xi_1|, \quad
 \quad L_2:=|\tau_2|-|\xi_2|^3-\frac{\a\b}{2}|\xi_2|,
\end{equation}
\begin{equation}\label{l2}
 L_3:=|\tau_3|-|\xi_3|^3-\frac{\a\b}{2}|\xi_3|,
\end{equation}
Since $\a\in (0,1]$ and the larger $ \a $ is, the more challenging the proof is, then we just need to consider the case when $\a=1$. In addition, we will choose $\beta=-1$ in the remainder of the proof since the case when $\b=1$ can be addressed similarly. Notice that $\la \xi_1\ra \la \xi_2\ra /\la \xi_3\ra \geq 1$, we thus  just need to address the case for $\rho$ is close to $\frac34$. Without loss of generality, we assume $\frac{9}{16} \leq \rho<\frac34$ and define $\si_0= \frac{7}{12}-\frac{\rho}{9}$ to help to justify the proof for $\frac12< \si\leq \si_0$.

To establish the estimate \eqref{bi1},  we will analyze the six possible cases for the sign of $\tau_1$, $\tau_2$ and $\tau_3$,
\begin{description}
\item[A] $\tau_1\geq 0$, $\tau_2\geq 0 $, $\tau_3\leq 0$;
\item[B] $\tau_1\geq 0$, $\tau_2\leq 0 $, $\tau_3\geq 0$;
\item[C] $\tau_1\geq 0$, $\tau_2\leq 0 $, $\tau_3\leq 0$;
\item[D] $\tau_1\leq 0$, $\tau_2\geq 0 $, $\tau_3\geq 0$;
\item[E] $\tau_1\leq 0$, $\tau_2\geq 0 $, $\tau_3\leq 0$;
\item[F] $\tau_1\leq 0$, $\tau_2\leq 0 $, $\tau_3\geq 0$.
\end{description}
We point out that cases $\tau_1,\tau_2,\tau_3\geq 0$ and $\tau_1,\tau_2,\tau_3\leq 0$ are impossible due to  $\tau_1+\tau_2+\tau_3=0$. For simplification, we notice that conditions $\mathbf{A}$, $\mb{B}$, $\mb{C}$ equivalent   $\mb{F}$, $\mb{E}$, $\mb{D}$ respectively, by considering $(\tau_1,\tau_2,\tau_3)\rightarrow -(\tau_1,\tau_2,\tau_3)$ and $L^2-$norms being invariant under reflections. Besides, because of the symmetric structural between $\tau_1$ and $\tau_2$ in \eqref{bi1}, cases $\mb{B}$ and $\mb{D}$ are also equivalence. Therefore, we will only need to justify the proof in cases $\mb{A}$ and $\mb{B}$. Further consideration will be needed for different conditions of $(\tau_1,\tau_2,\tau_3)$, we  split the space of $(\xi_1,\xi_2,\xi_3)\in\R^3$ as following:
\begin{description}
  \item[(a)] $\xi_1\geq 0$, $\xi_2\geq0$, $\xi_3\leq 0$;
  \item[(b)] $\xi_1\geq 0$, $\xi_2\leq0$, $\xi_3\leq 0$;
  \item[(c)] $\xi_1\leq 0$, $\xi_2\geq0$, $\xi_3\leq 0$;
  \item[(d)] $\xi_1\leq 0$, $\xi_2\leq0$, $\xi_3\geq 0$;
  \item[(e)] $\xi_1\leq 0$, $\xi_2\geq0$, $\xi_3\geq 0$;
  \item[(f)] $\xi_1\geq 0$, $\xi_2\leq0$, $\xi_3\geq 0$.
\end{description}
 It then suffices to consider   cases $\mb{(a)}$, $\mb{(b)}$ and $\mb{(c)}$. We first start to consider case  $\mb{A}$.  According to \eqref{l1} and \eqref{l2}, we can write
\[\ L_1= \tau_1-|\xi_1|^3+\frac12|\xi_1|; \ \ L_2= \tau_2-|\xi_2|^3+\frac12|\xi_2|;\ \  L_3:= \tau_3+|\xi_3|^3-\frac12|\xi_3|.\]
Here, we point out that $L_3$ should be in the form of $- \l(\tau_3+|\xi_3|^3-\frac12|\xi_3|\r)$ based on \eqref{l2}. Since the negative sign  in $\la L_3\ra $ does not influence  the computation,   for convenience of the later proof, we will use the representation for  $L_3$ as shown above.  Similar issue in notations  will be brought out in the discussion of Case $\mb{B}$.
\begin{itemize}
  \item{\bf Case A1.} $|\xi_1|\leq 1$.\\
  In this region, one has $\la \xi_1\ra \la \xi_2\ra /\la \xi_3\ra \lesssim 1$. Hence,
  \begin{align*}
    \mbox{LHS of \eqref{bi1}}
    \lesssim & \int_A \frac{|\xi_3|^2|f_1(\xi_1,\tau_1)f_2(\xi_2,\tau_2)f_3(\xi_3,\tau_3)|}{\phi(\xi_3)\la \xi_3 \ra^{-2} \la L_1\ra^{\si} \la L_2\ra^\frac12 \la L_3\ra^{1-\si}}
    \end{align*}

 Notice that $\phi(\xi)=|\xi|\sqrt{\xi^4+\xi^2+1}\sim |\xi|\la \xi\ra^2$. Thus, \eqref{bi1} reduces to
 \begin{equation}\label{bi2}
   \int_A \frac{|\xi_3|\prod^3_{j=1}|f_j(\xi_j,\tau_j)|}{ \la L_1\ra^{\si} \la L_2\ra^\frac12 \la L_3\ra^{1-\si}}\lesssim \prod^3_{j=1}\|f_j\|_{L^2_{\xi, \tau}}
 \end{equation}
  \item{\bf Case A1.1.} If $|\xi_2|\leq 2$, it follows that $|\xi_3|\leq 3$.

      Thus,
      \begin{align*}
        \mbox{LHS of \eqref{bi2}}\lesssim    & \iint \frac{|f_3|}{\la L_3\ra^{1-\si}} \l(\iint\frac{|f_1f_2|}{\la L_1\ra^{\si} \la L_2\ra^{\frac12}}d\xi_1 d\tau_1\r)d\xi_3 d\tau_3\\
  \lesssim & \iint \frac{|f_3|}{\la L_3\ra^{1-\si}} \l(\iint\frac{d\xi_1 d\tau_1}{\la L_1\ra^{2\si} \la L_2\ra}\r)^\frac12\l(\iint f_1^2f_2^2 d\xi_1 d\tau_1\r)^\frac12d\xi_3 d\tau_3.
      \end{align*}
  According to Lemma \ref{lem1} and $\tau_1+\tau_2+\tau_3=0$, we obtain that, for $|\xi_1|\leq 1$ and $\frac12<\si<1$,
  \[\sup_{\xi_3,\tau_3} \frac{1}{\la L_3\ra^{2-2\si}}\int \frac{d\xi_1}{\la L_1 +L_2\ra}\lesssim 1.\]
  Therefore, \eqref{bi2} is achieved through Cauchy-Schwartz inequality,
  \begin{align*}
     \mbox{LHS of \eqref{bi2}}\lesssim    &\iint |f_3| \l(\iint f_1^2f_2^2 d\xi_1 d\tau_1\r)^\frac12d\xi_3 d\tau_3\\
     \lesssim & \prod^3_{j=1}\|f_j\|_{L^2_{\xi_\tau}}.
  \end{align*}

  \item{\bf Case A1.2.} If $|\xi_2|>2$ and $\la L_2\ra\leq \la L_3\ra$, this implies that, for $\si>\frac12$,
      \[\la L_2\ra^{\frac12} \la L_3\ra^{1-\si}\geq \la L_2\ra^\si\la L_3\ra^{\frac32-2\si}.\]
  Thus,
  \[\mbox{LHS of \eqref{bi2}}\lesssim \int_A \frac{|\xi_3|\prod^3_{j=1}|f_j|}{\la L_1\ra^\si \la L_2\ra^\si\la L_3\ra^{\frac32-2\si}}.\]
  It then suffices to bound
 \begin{equation}\label{bi2-1}
  \sup_{\xi_3,\tau_3}\frac{|\xi_3|^2}{\la L_3\ra^{3-4\si}}\int\frac{d\xi_1}{\la L_1 +L_2\ra^{2\si}}.
 \end{equation}
 In addition, we point out that $|\xi_1|\leq 1$ and $|\xi_2|>2$ will lead to $|\xi_2|\sim |\xi_3|$ which will become useful in the later study.
  \begin{itemize}
    \item{\bf Case A1.2.1} If $(\xi_1,\xi_2,\xi_3)$ lies in $\mb{(a)}$, that is, $\xi_1\geq 0$, $\xi_2\geq0$, $\xi_3\leq 0$. Then,
        \[ L_1+L_2=  3\xi_3\xi_1^2 +3\xi_3^2\xi_1 -L_3, \quad \mbox{with}\quad   L_3= \tau_3-\xi_3^3+\frac12\xi_3.\]
       For $\si\leq \si_0=\frac{7}{12}-\frac{\rho}{9}<\frac58$,  it yields  $3-4\si>\frac12$. Then, according to Lemma \ref{lem3}, one has
        \begin{align*}
          \mbox{\eqref{bi2-1}} \lesssim  \frac{|\xi_3|^2}{\la L_3\ra^{3-4\si}} \frac{1}{|\xi_3|^{\frac12}} \la L_3-\frac34\xi_3^3\ra^{-\frac12}\lesssim \frac{|\xi_3|^\frac32}{\la L_3\ra^{3-4\si}\la L_3-\frac34\xi_3^3\ra^{\frac12}}\lesssim 1,
        \end{align*}
        since $\la L_3\ra^{3-4\si}\la L_3-\xi_3^3\ra^{\frac12}\gtrsim \la L_3\ra^\frac12 \la L_3-\xi_3^3\ra^{\frac12}\gtrsim \la \xi_3^3\ra^\frac12$. 

    \item{\bf Case A1.2.2} If $(\xi_1,\xi_2,\xi_3)$ lies in $\mb{(b)}$, that is, $\xi_1\geq 0$, $\xi_2\leq0$, $\xi_3\leq 0$. Then,
        \[ L_1= \tau_1-\xi_1^3+\frac12\xi_1,\quad L_2= \tau_2+\xi_2^3-\frac12\xi_2  ,\quad L_3= \tau_3-\xi_3^3+\frac12\xi_3 \] and
          \[ L_1+L_2+L_3=\xi_2(3\xi_1^2+3\xi_1\xi_2+2\xi_2^2-1).\]
          Hence, \[|L_1+L_2+L_3|\gtrsim |\xi_2||\xi_1^2+\xi_2^2|\sim |\xi_2|^3\sim|\xi_{3}|^3,\]
          since $|\xi_1|\leq 1$, $|\xi_2|>2$ and $|\xi_2|\sim|\xi_3|$.
          In addition, we notice that $\si<\frac{7}{12}-\frac{\rho}{9}<\frac{7}{12}$  leads to $3-4\si>\frac23$, thus,
          \[\la L_3\ra^{3-4\si}\la L_1+L_2\ra^{2\si}\gtrsim \la L_1+L_2+L_3\ra^{\frac23}\sim \la \xi_3\ra^2.\]
          
        Therefore,   one has
        \[\mbox{\eqref{bi2-1}} \lesssim \frac{|\xi_3|^2}{\la \xi_3 \ra^2}   \int_{|\xi_1|\leq 1} d\xi_1\lesssim 1,\]
        
    \item{\bf Case A1.2.3} If $(\xi_1,\xi_2,\xi_3)$ lies in $\mb{(c)}$, that is, $\xi_1\leq 0$, $\xi_2\geq0$, $\xi_3\leq 0$. Then,
        \[ L_1= \tau_1+\xi_1^3-\frac12\xi_1,\quad L_2= \tau_2-\xi_2^3+\frac12\xi_2  ,\quad L_3= \tau_3-\xi_3^3+\frac12\xi_3 \] and
       \[ L_1+L_2+L_3=\xi_1(2\xi_1^2+3\xi_1\xi_2+3\xi_2^2-1):=P(\xi_1),\]
       with
       \[P'(\xi_1)=6\xi_1^2+6\xi_1\xi_2+3\xi_2^2-1.\]
       We notice  that, for $|\xi_1|\leq 1$, $|\xi_2|>2$, one has $$|P'(\xi_1)|\gtrsim |\xi_2|^2.$$ 
  Therefore, we obtain that 
        \[\mbox{\eqref{bi2-1}} \lesssim \frac{|\xi_3|^2}{\la L_3\ra^{3-4\si}}\int\frac{1}{|P'(\xi_1)|}\frac{|P'(\xi_1)|}{\la P(\xi_{1})\ra^{2\si}}d\xi_1\lesssim \int\frac{|P'(\xi_1)|}{\la P(\xi_{1})\ra^{2\si}}d\xi_1\lesssim 1,\]
        since $|\xi_2|\sim |\xi_3|$ and $\si>\frac12$.
  \end{itemize}

  \item{\bf Case A1.3.} If $|\xi_2|>2$, $\la L_2\ra> \la L_3\ra$. Similar to {\bf Case A1.2}, one still has
      \[|\xi_2|\sim |\xi_3|\gtrsim|\xi_1|.\]
  In addition, for $\si>\frac12$,
     \[\la L_2\ra^{\frac12} \la L_3\ra^{1-\si}\geq \la L_2\ra^{\frac32-2\si}\la L_3\ra^\si.\]
     Thus,
     \[\mbox{\eqref{bi2-1}}\lesssim \iint\frac{|\xi_2||f_2|}{\la L_2\ra^{\frac32-2\si}} \l(\iint\frac{|f_1f_3|}{\la L_1\ra^\si \la L_3\ra^\si}d\xi_1d\tau_1\r)d\xi_2d\tau_2,\]
     which reduces the proof similar to \textbf{Case A1.2}.
   \item{\bf Case A2.} $|\xi_2|\leq 1$. This case can be addressed same as \textbf{Case A1}.

   \item{\bf Case A3.} $|\xi_1|, |\xi_2|> 1$ and $\la L_1\ra=\mbox{MAX}:=\max{\{\la L_1\ra, \la L_2\ra ,\la L_3\ra\}}$.

       This follows that
       \[\la L_1\ra^\frac12 \la L_2\ra^\frac12 \la L_3\ra^{1-\si}\geq \la L_1\ra^{2-3\si} \la L_2\ra^\si \la L_3\ra^{\si}.\]
Hence, it leads to,
\begin{align*}
\mbox{LHS of \eqref{bi1}} \lesssim & \int_A \frac{|\xi_3|\la \xi_1\ra^\rho \la \xi_2\ra^\rho |f_1f_2f_3|}{\la \xi_3\ra^{\rho}\la L_1\ra^{2-3\si} \la L_2\ra^\si \la L_3\ra^{\si}}\\
\lesssim & \iint \frac{\la \xi_1\ra^\rho |f_1|}{\la L_1\ra^{2-3\si}}\iint \frac{|\xi_3|\la \xi_2\ra^\rho |f_2f_3|d\xi_2d\tau_2}{\la \xi_3\ra^{\rho}\la L_2\ra^\si \la L_3\ra^{\si}}d\xi_1d\tau_1.
\end{align*}
It then suffices to bound the term
\begin{equation}\label{L11}
\sup_{\xi_1,\tau_1}\frac{\la \xi_1\ra^{2\rho}}{\la L_1\ra^{4-6\si}}\int\frac{|\xi_3|^2\la \xi_2\ra^{2\rho}}{\la \xi_3\ra^{2\rho}\la L_2+L_3\ra^{2\si}}  d\xi_2,
\end{equation}
with considering $(\xi_1,\tau_1)$ fixed in the integral \eqref{L11} and
\begin{align*}
L_2+L_3 =& \tau_2-|\xi_2|^3+\frac12|\xi_2|+\tau_3+|\xi_3|^3-\frac12|\xi_3|\nonumber\\
=& -|\xi_2|^3+\frac12|\xi_2|+|\xi_3|^3-\frac12|\xi_3|-\tau_1.
\end{align*}
We establish the estimate \eqref{L11} in following cases.
\end{itemize}
\begin{itemize}
  \item{\bf Case A3.1.} If $(\xi_1,\xi_2,\xi_3)$ lies in $\mb{(a)}$, that is, $\xi_1\geq 0$, $\xi_2\geq0$, $\xi_3\leq 0$.

  In this case, we  have,
\begin{align*}
  Q_1(\xi_2):=&L_2+L_3=-\xi_1^3+\frac12\xi_1-\xi_3^3+\frac12\xi_3-\tau_1\nonumber\\
  =& -\xi_1^3+\frac12\xi_1+(\xi_1+\xi_2)^3-\frac12(\xi_1+\xi_2)-\tau_1\nonumber\\
  =& 3\xi_1\xi_2^2+3\xi_1^2\xi_2-\tau_1+\xi_1^3-\frac12\xi_1,
\end{align*}
and
\be\label{Q-1}
Q_1'(\xi_2)=3\xi_1(2\xi_2+\xi_1), 
\ee
where
$|Q_1'(\xi_2)|\gtrsim |\xi_1\xi_2|$ since $\xi_1,\xi_2>1$.
In addition, one also has
\[|L_1+L_2+L_3|= |3\xi_1\xi_2(\xi_1+\xi_2)|\gtrsim |\xi_1\xi_2\xi_3|,\]
which implies that $\la L_1 \ra \gtrsim \la L_1+L_2+L_3\ra \gtrsim \la\xi_1\xi_2\xi_3 \ra$.
Thus, the term \eqref{L11} is bounded by
\begin{equation}\label{L11-1}
\sup_{\xi_1,\tau_1}\frac{\la \xi_1\ra^{2\rho}}{\la L_1\ra^{4-6\si}}\int\frac{|\xi_3|^2\la \xi_2\ra^{2\rho}}{| \xi_3|^{2\rho} |\xi_1\xi_2|} \frac{|Q_1'(\xi_2)|}{\la Q_1(\xi_2)\ra^{2\si}}  d\xi_2.
\end{equation}
We notice that, for $\rho\leq \frac34$, and $|\xi_1|, |\xi_2|>1$,
$$  \la \xi_1\ra^{2\rho-1}\la \xi_2\ra^{2\rho-1}| \xi_3|^{-2\rho+2}\lesssim |\xi_1\xi_2\xi_3|^{2-2\rho}\lesssim \la \xi_1\xi_2\xi_3\ra^{2-2\rho}\lesssim \la L_1 \ra^{2-2\rho}.$$
This leads to the term \eqref{L11-1} is bounded by
\[ \sup_{\xi_1,\tau_1}\frac{\la L_1 \ra^{2-2\rho}}{\la L_1\ra^{4-6\si}}\int \frac{|Q_1'(\xi_2)|}{\la Q_1(\xi_2)\ra^{2\si}}  d\xi_2\lesssim \sup_{\xi_1,\tau_1}\frac{\la L_1 \ra^{2-2\rho}}{\la L_1\ra^{4-6\si}},\]
which implies that \eqref{L11} is finite since $2-2\rho\leq 4-6\si$ with $\si\leq \si_0=\frac{7}{12}-\frac{\rho}{9}$ and $\rho\geq \frac{9}{16}$.

  \item{\bf Case A3.2.} If $(\xi_1,\xi_2,\xi_3)$ lies in $\mb{(b)}$, that is, $\xi_1\geq 0$, $\xi_2\leq0$, $\xi_3\leq 0$.

      In this case, we  have
\begin{equation*}
  Q_2(\xi_2):=L_2+L_3=2\xi_2^3+3\xi_1\xi_2^2+(3\xi_1^2-\frac12)\xi_2-\tau_1+\xi_1^3-\frac12\xi_1,
\end{equation*}
and
\be\label{Q-2}
Q_2'(\xi_2)=6\xi_2^2+6\xi_1\xi_2+3\xi_1^2-\frac12.
\ee
We notice that $|Q_2'(\xi_2)|\gtrsim \xi_1^2+\xi_2^2\gtrsim|\xi_1\xi_2|$ since $ |\xi_2|>1 $.
In addition, one also has
\[|L_1+L_2+L_3|= |\xi_2 (2\xi_2^2+3\xi_1\xi_2+3\xi_1^2-1)|\gtrsim |\xi_1\xi_2\xi_3|,\]
since $|\xi_1|, |\xi_2|>1$. It follows that $\la L_1 \ra \gtrsim \la \xi_1\xi_2\xi_3 \ra$. Hence, the remainder proof is similar to  the line in \textbf{Case A3.1}.

  \item{\bf Case A3.3.} If $(\xi_1,\xi_2,\xi_3)$ lies in $\mb{(c)}$, that is, $\xi_1\leq 0$, $\xi_2\geq0$, $\xi_3\leq 0$.

      In this case, we  have
\begin{equation*}
  Q_3(\xi_2):=L_2+L_3=3\xi_1\xi_2^2+3\xi_1^2\xi_2-\tau_1+\xi_1^3-\frac12\xi_1,
\end{equation*}
and
\be\label{Q-3}
Q_3'(\xi_2)=3\xi_1(2\xi_2+\xi_1).
\ee
In addition, one also has
\[|L_1+L_2+L_3|= |\xi_1(2\xi_2^2+3\xi_1\xi_2+3\xi_1^2-1)|\gtrsim |\xi_1\xi_2\xi_3|,\]
since $|\xi_1|, |\xi_2|>1$. It follows that $\la L_1 \ra \gtrsim \la \xi_1\xi_2\xi_3 \ra$. %
\begin{itemize}
  \item{\bf Case A3.3.1.} If $|2\xi_2+\xi_1|\geq \frac{1}{10}|\xi_2|$, then
  $|Q_3'(\xi_2)|\gtrsim |\xi_1\xi_2|$. Thus,  the rest of the proof is similar to A3.1.

  \item{\bf Case A3.3.2.} If $|2\xi_2+\xi_1|\leq \frac{1}{10}|\xi_2|$, it implies that $$|\xi_1|\sim |\xi_2|\sim |\xi_3|.$$
      Then, \eqref{L11} is bounded by
      \begin{align}
        \sup_{\xi_1,\tau_1}\frac{\la \xi_1\ra^{\frac12}}{\la L_1\ra^{4-6\si}}\int\frac{|\xi_2|^2\la \xi_1\ra^{2\rho-\frac12}\la \xi_2\ra^{2\rho}}{| \xi_3|^{2\rho}}\frac{1}{\la Q_3(\xi_2)\ra^{2\si}}  d\xi_2\label{L11-2}.
      \end{align}
      We notice that
      \[| \xi_1|^{2\rho-\frac12}|\xi_2|^{2\rho}| \xi_3|^{-2\rho+2}\sim |\xi_1|^{2\rho+\frac32}\lesssim | \xi_1\xi_2\xi_3|^{\frac{2\rho}{3}+\frac{1}{2}}\lesssim \la L_1\ra^{\frac{2\rho}{3}+\frac{1}{2}}.\]
      In addition, according to Lemma \ref{lem2}, one has
      \[\int\frac{1}{\la Q_3(\xi_2)\ra^{2\si}}d\xi_2 \lesssim |\xi_1|^{-\frac12}.\]
      Hence, the term \eqref{L11-2} is bounded by
      \[\sup_{\xi_1,\tau_1}\frac{\la L_1\ra^{\frac{2\rho}{3}+\frac{1}{2}}}{\la L_1\ra^{4-6\si}},\]
      which leads to the boundedness of  \eqref{L11}  since $\frac{2\rho}{3}+\frac{1}{2} \leq 4-6\si$.
\end{itemize}
\vskip\baselineskip
\item{\bf Case A4.} $|\xi_1|, |\xi_2|>1$ and $\la L_2\ra=\mbox{MAX}$. Due to the symmetric structure between $\xi_1$ and $\xi_2$ as well as their representation of $L_1$ and $L_2$. The proof is same as \textbf{Case A3}.
     \vskip\baselineskip
\item{\bf Case A5.} $|\xi_1|, |\xi_2|> 1$ and $\la L_3\ra=\mbox{MAX}$.

    This follows that
       \[\la L_1\ra^\frac12 \la L_2\ra^\frac12 \la L_3\ra^{1-\si}\geq \la L_3\ra^{2-3\si} \la L_1\ra^\si \la L_2\ra^{\si}.\]
Hence, it leads to,
\begin{align*}
\mbox{LHS of \eqref{bi1}}
\lesssim & \iint \frac{ |\xi_3||f_3|}{\la \xi_3\ra^{\rho}\la L_3\ra^{2-3\si}}\iint \frac{\la \xi_1\ra^\rho\la \xi_2\ra^\rho |f_1f_2|d\xi_1d\tau_1}{ \la L_1\ra^\si \la L_2\ra^{\si}}d\xi_3d\tau_3.
\end{align*}
In order to establish \eqref{bi1},  it then suffices to bound the term
\begin{equation}\label{L13}
\sup_{\xi_3,\tau_3}\frac{|\xi_3|^2}{\la \xi_3\ra^{2\rho}\la L_3\ra^{4-6\si}}\int\frac{\la \xi_1\ra^{2\rho}\la \xi_2\ra^{2\rho}}{ \la L_1+L_2\ra^{2\si}}  d\xi_1,
\end{equation}
with considering $(\xi_3,\tau_3)$ fixed in the integral \eqref{L13} and
\begin{align*}
L_1+L_2=&\tau_1-|\xi_1|^3+\frac12|\xi_1|+\tau_2-|\xi_2|^3+\frac12|\xi_2|\nonumber\\
=& -|\xi_1|^3+\frac12|\xi_1|-|\xi_2|^3+\frac12|\xi_2|-\tau_3.
\end{align*}
We then consider to bound the term \eqref{L13} in following cases.
\end{itemize}
\begin{itemize}
  \item{\bf Case A5.1.} If $(\xi_1,\xi_2,\xi_3)$ lies in $\mb{(a)}$, that is, $\xi_1\geq 0$, $\xi_2\geq0$, $\xi_3\leq 0$.

      In this case, we  have
\begin{align*}
  Q_4(\xi_1):=&L_1+L_2=-\xi_1^3+\frac12\xi_1-\xi_2^3+\frac12\xi_2-\tau_3\nonumber\\
  =& 3\xi_3\xi_1^2+3\xi_3^2\xi_1-\tau_3+\xi_3^3-\frac12\xi_3,
\end{align*}
and
\be\label{Q-4}
Q_4'(\xi_1)=3\xi_3(2\xi_1+\xi_3).
\ee
In addition, one also has
\[|L_1+L_2+L_3|= |3\xi_1\xi_2(\xi_1+\xi_2)|\gtrsim|\xi_1\xi_2\xi_3|,\]
which follows that $\la L_3 \ra \gtrsim \la \xi_1\xi_2\xi_3 \ra$.
\begin{itemize}
  \item{\bf Case A5.1.1.} If $|\xi_1|\leq \frac{1}{10}|\xi_3|$, then $|\xi_1| \lesssim |\xi_3|\sim |\xi_2|$ and
  $|Q_4'(\xi_1)|\gtrsim |\xi_3|^2$. Thus, the term \eqref{L13} is bounded by
       \begin{equation}\label{L13-1}
        \sup_{\xi_3,\tau_3}\frac{|\xi_3|^2}{\la \xi_3\ra^{2\rho}\la L_3\ra^{4-6\si}}\int\frac{\la \xi_1\ra^{2\rho}\la \xi_2\ra^{2\rho}}{|\xi_3|^2}\frac{|Q'_4(\xi_1)|}{\la Q_4(\xi_1)\ra^{2\si}}  d\xi_1.
      \end{equation}
       We notice that, for $\rho>0$,
 $$  | \xi_1|^{2\rho}| \xi_2|^{2\rho}| \xi_3|^{-2\rho} \lesssim | \xi_1|^{2\rho}\lesssim | \xi_1\xi_2\xi_3|^{\frac{2\rho}{3}}\lesssim \la L_3 \ra^{\frac{2\rho}{3}}.$$
 Thus, the term \eqref{L13-1} is bounded by
 \[\sup_{\xi_3,\tau_3}\frac{ \la L_3 \ra^{\frac{2\rho}{3}}}{\la L_3\ra^{4-6\si}},\]
 which  is finite since $\frac{2\rho}{3}\leq 4-6\si$.
  \item{\bf Case A5.1.2.} If $ \frac{1}{10}|\xi_3|\leq |\xi_1|\leq 10|\xi_3|$, then $|\xi_2| \lesssim |\xi_1|\sim |\xi_3|$. Thus, the term \eqref{L13} is bounded by
       \begin{equation}\label{L13-2}
        \sup_{\xi_3,\tau_3}\frac{| \xi_3|^{\frac12}}{\la L_3\ra^{4-6\si}}\int\frac{| \xi_3|^{\frac32}\la \xi_1\ra^{2\rho}\la \xi_2\ra^{2\rho}}{| \xi_3|^{2\rho}}\frac{1}{\la Q_4(\xi_1)\ra^{2\si}}  d\xi_1.
      \end{equation}
      We notice that, for $\rho>0$,
 $$  \la \xi_1\ra^{2\rho}\la \xi_2\ra^{2\rho}\la \xi_3\ra^{-2\rho+\frac32}  \lesssim | \xi_1\xi_2\xi_3|^{\frac{2\rho}{3}+\frac12}\lesssim \la L_3 \ra^{\frac{2\rho}{3}+\frac12},$$
 since $2\rho \geq \frac23 \rho+\frac12$.
   In addition, according to Lemma \ref{lem2}, one has
 \[\int\frac{1}{\la Q_4(\xi_1)\ra^{2\si}}d\xi_1 \lesssim |\xi_3|^{-\frac12}.\]
  Thus, the term \eqref{L13-2} is bounded for $\frac{2\rho}{3}+\frac12\leq 4-6\si$.
  \item{\bf Case A5.1.3.}  If $  |\xi_1|\geq 10|\xi_3|$, then $|\xi_3| \lesssim |\xi_1|\sim |\xi_2|$ and $|Q_4'(\xi_1)|\gtrsim |\xi_1\xi_3|$.
      Thus, the term \eqref{L13} is bounded by
           \begin{equation}\label{L13-3}
        \sup_{\xi_3,\tau_3}\frac{|\xi_3|^2}{\la \xi_3\ra^{2\rho}\la L_3\ra^{4-6\si}}\int\frac{\la \xi_1\ra^{2\rho}\la \xi_2\ra^{2\rho}}{|\xi_1\xi_3|}\frac{|Q'_4(\xi_1)|}{\la Q_4(\xi_1)\ra^{2\si}}  d\xi_1.
      \end{equation}
            We notice that, for $\frac{9}{16}<\rho<\frac34$,  one has
            \[|\xi_3| \lesssim |\xi_3|^{2\rho-\frac12}\la \xi_3\ra^{2\rho}.\]
            Hence, for $|\xi_1|$, $|\xi_2|>1$, it yields,
        \[\la \xi_1\ra^{2\rho-1}\la \xi_2\ra^{2\rho}\la \xi_3\ra^{-2\rho}|\xi_3|\lesssim |\xi_1 \xi_2 \xi_3|^{2\rho-\frac12}\lesssim \la L_3\ra^{2\rho-\frac12}.\]
        Therefore, \eqref{L13-3} is finite if $2\rho-\frac12\leq 4-6\si$.
 \end{itemize}

  \item{\bf Case A5.2.} If $(\xi_1,\xi_2,\xi_3)$ lies in $\mb{(b)}$, that is, $\xi_1\geq 0$, $\xi_2\leq0$, $\xi_3\leq 0$.

      In this case, we  have
\begin{align*}
  Q_5(\xi_1):=&L_1+L_2=-\xi_1^3+\frac12\xi_1+\xi_2^3-\frac12\xi_2-\tau_3\nonumber\\
  =& -2\xi_1^3-3\xi_3\xi_1^3-(3\xi_3^2-1)\xi_1-\tau_3-\xi_3^3+\frac12\xi_3,
\end{align*}
and
\be\label{Q-5}
Q_5'(\xi_1)=-6\xi_1^2-6\xi_1\xi_3-3\xi_3^2+1.
\ee
We notice that $|Q_5'(\xi_1)|\gtrsim |\xi_1|^2+|\xi_1+\xi_3|^2\gtrsim |\xi_1\xi_2|$ since  $|\xi_1| >1$.
In addition, one also has
\begin{align*}
|L_1+L_2+L_3|=& |-(\xi_1+\xi_3)(2\xi_1^2+\xi_1\xi_3+2\xi_3^2-1)|\\
=& |\xi_2(2\xi_1^2+\xi_1\xi_3+2\xi_3^2-1)|\gtrsim |\xi_1\xi_2\xi_3|,
\end{align*}
since $|\xi_1| >1$. It  follows that $\la L_3 \ra \gtrsim \la \xi_1\xi_2\xi_3 \ra$. Then, the remaining proof is similar to the \textbf{Case A3.1}.

  \item{\bf Case A5.3.} If $(\xi_1,\xi_2,\xi_3)$ lies in $\mb{(c)}$, that is, $\xi_1\leq 0$, $\xi_2\geq0$, $\xi_3\leq 0$.

      In this case, we  have
\begin{align*}
  Q_6(\xi_1):=&L_1+L_2=\xi_1^3-\frac12\xi_1-\xi_2^3+\frac12\xi_2-\tau_3\nonumber\\
  =& 2\xi_1^3+3\xi_3\xi_1^2+(3\xi_3^2-1)\xi_1-\tau_3+\xi_3^3-\frac12\xi_3,
\end{align*}
and
\be\label{Q-6}
Q_6'(\xi_1)=6\xi_1^2+6\xi_1\xi_3+3\xi_3^2-1.
\ee
We notice that $|Q_6'(\xi_1)|\gtrsim |\xi_1|^2+|\xi_1+\xi_3|^2\gtrsim |\xi_1\xi_2|$ since  $|\xi_1| >1$.
In addition, one also has
\begin{align*}
|L_1+L_2+L_3|=& |2\xi_1^3+3\xi_1^2\xi_3+3\xi_1\xi_3^2-\xi_1|\\
=& |\xi_1(2\xi_1^2+\xi_1\xi_3+2\xi_3^2-1)|\gtrsim |\xi_1\xi_2\xi_3|.
\end{align*}
which follows that $\la L_3 \ra \gtrsim \la \xi_1\xi_2\xi_3 \ra$ since $|\xi_1| >1$. Then, the remaining proof is similar to the \textbf{Case 3.1}.
\end{itemize}
\vskip\baselineskip
Next, we move on to consider the case that $(\tau_1,\tau_2,\tau_3)$ lies in region \textbf{B}, and  write\footnote{As we mentioned earlier, we neglect the ``$-$" in the  representation of $L_2$.}
\[L_1=\tau_1-|\xi_1|^3+\frac12|\xi_1|,\quad L_2=\tau_2+|\xi_2|^3-\frac12|\xi_2|,\quad L_3=\tau_3-|\xi_3|^3+\frac12|\xi_3|.\]
\begin{itemize}
  \item{\bf Case B1.} $|\xi_1|\leq 1$. Similar to the discussion in \textbf{Case A1}, it suffices to establish
      \be\label{bi3}
         \int_A \frac{|\xi_3|\prod^3_{j=1}|f_j(\xi_j,\tau_j)|}{ \la L_1\ra^{\si} \la L_2\ra^\frac12 \la  L_3\ra^{1-\si}}\lesssim \prod^3_{j=1}\|f_j\|_{L^2_{\xi_\tau}}.
      \ee
  \item{\bf Case B1.1} If $|\xi_2|\leq 2$. The proof of this case is quite similar to \textbf{Case A1.1}, therefore omitted.
   \item{\bf Case B1.2} If $|\xi_2|>2$,  $\la L_2\ra \leq \la L_3\ra$. This follows that, for $\si>\frac12$.
        \[\la L_2\ra^{\frac12} \la L_3\ra^{1-\si}\geq \la L_2\ra^\si\la L_3\ra^{\frac32-2\si}.\]
        Thus, similar to \textbf{Case A1.2}, one has $|\xi_2|\sim|\xi_3|$ and it suffices to bound
     \begin{equation}\label{bi4}
        \sup_{\xi_3,\tau_3} \frac{|\xi_3|^2}{\la L_3 \ra^{3-4\si}}\int \frac{d\xi_1}{\la L_1+L_2\ra^{2\si}}.
     \end{equation}

%
%

        \begin{itemize}

    \item{\bf Case B1.2.1} If $(\xi_1,\xi_2,\xi_3)$ lies in $\mb{(a)}$, that is, $\xi_1\geq 0$, $\xi_2\geq0$, $\xi_3\leq 0$. Then,
     \[ L_1= \tau_1-\xi_1^3+\frac12\xi_1,\quad L_2= \tau_2+\xi_2^3-\frac12\xi_2  ,\quad L_3= \tau_3+\xi_3^3-\frac12\xi_3 \] and
    \[ L_1+L_2+L_3=-\xi_1(2\xi_1^2+3\xi_1\xi_2+3\xi_2^2-1):=P_2(\xi_1)\]
     with
    \[P_2'(\xi_1)=-6\xi_1^2-6\xi_1\xi_2-3\xi_2^2+1.\]
     We notice  that $$|P_2'(\xi_1)|\gtrsim |\xi_2|^2,$$ 
    since $|\xi_1|\leq 1$, $|\xi_2|>2$. Therefore,  similar to \textbf{Case A1.2.3}, one has
    \[\mbox{\eqref{bi4}} \lesssim \frac{|\xi_3|^2}{\la L_3\ra^{3-4\si}}\int\frac{1}{|P'_2(\xi_1)|}\frac{|P'_2(\xi_1)|}{\la P_2(\xi_{1})\ra^{2\si}}d\xi_1 \lesssim 1.\]
   
    \item{\bf Case B1.2.2} If $(\xi_1,\xi_2,\xi_3)$ lies in $\mb{(b)}$, that is, $\xi_1\geq 0$, $\xi_2\leq0$, $\xi_3\leq 0$.
        Then,
        \[ L_1= \tau_1-\xi_1^3+\frac12\xi_1,\quad L_2= \tau_2-\xi_2^3+\frac12\xi_2  ,\quad L_3= \tau_3+\xi_3^3-\frac12\xi_3 \] and
        \[ |L_1+L_2+L_3|=|\xi_3(2\xi_3^2+3\xi_1\xi_3+3\xi_1^2-1)|\sim |\xi_3|^3 .\]
        Hence, similar to \textbf{Case A1.2.2}, one has
        \[\la L_3\ra^{3-4\si}\la L_1+L_2\ra^{2\si}\gtrsim \la L_1+L_2+L_3\ra^{\frac23}\]
        and
              \begin{align*}
          \mbox{\eqref{bi4}} \lesssim  \int_{|\xi_1|\leq 1} \frac{|\xi_3|^2}{\la L_1+L_2+L_3\ra^{\frac23}} d
          \xi_1 \lesssim 1.
        \end{align*}
   
    \item{\bf Case B1.2.3} If $(\xi_1,\xi_2,\xi_3)$ lies in $\mb{(c)}$, that is, $\xi_1\leq 0$, $\xi_2\geq0$, $\xi_3\leq 0$. Then,
        \[ L_1+L_2=  -3\xi_3\xi_1^2-3\xi_3^2\xi_1 -L_3, \quad L_3= \tau_3+\xi_3^3-\frac12\xi_3.\] 
       Then, we can just repeat the proof in Case \textbf{A1.2.1}.
  \end{itemize}
    \item{\bf Case B1.3} If $|\xi_2|> 2$, $\la L_2 \ra >\la L_3\ra$, then, for $|\xi_1|\leq 1$, one has $|\xi_2|\sim |\xi_3|$. Thus, it then suffices to bound
    \begin{equation*}
        \sup_{\xi_2,\tau_2} \frac{|\xi_2|^2}{\la L_2 \ra^{3-4\si}}\int \frac{d\xi_1}{\la L_1+L_3\ra^{2\si}},
     \end{equation*}
        which is similar to \textbf{Case B1.2}. therefore omitted. \vskip\baselineskip
  \item{\bf Case B2.} $|\xi_2|\leq 1$. The proof for this case is also similar to \textbf{Case A1} and \textbf{Case B1}.

      \vskip\baselineskip
  \item{\bf Case B3.} $|\xi_1|$, $|\xi_2|>1$ and $\la L_1\ra=\mbox{MAX}$.\\
      Similar to the \textbf{Case A3}, it suffices to bound
      \begin{equation}\label{L21}
\sup_{\xi_1,\tau_1}\frac{\la \xi_1\ra^{2\rho}}{\la L_1\ra^{4-6\si}}\int\frac{|\xi_3|^2\la \xi_2\ra^{2\rho}}{\la \xi_3\ra^{2\rho}\la L_2+L_3\ra^{2\si}}  d\xi_2,
\end{equation}
with considering $(\xi_1,\tau_1)$ fixed in the integral \eqref{L21} and
\begin{align*}
L_2+L_3=&\tau_2+|\xi_2|^3-\frac12|\xi_2|+\tau_3-|\xi_3|^3+\frac12|\xi_3|\nonumber\\
=& |\xi_2|^3-\frac12|\xi_2|-|\xi_3|^3+\frac12|\xi_3|-\tau_1.
\end{align*}
\end{itemize}
\begin{itemize}
  \item{\bf Case B3.1.} If $(\xi_1,\xi_2,\xi_3)$ lies in $\mb{(a)}$, that is, $\xi_1\geq 0$, $\xi_2\geq0$, $\xi_3\leq 0$.

      In this case, we  have
\begin{equation*}
  Q_7(\xi_2):=L_2+L_3=-3\xi_1\xi_2^2-3\xi_1^2\xi_2-\tau_1-\xi_1^3
  +\frac12\xi_1,
\end{equation*}
and
\be\label{Q-7}
Q_7'(\xi_2)=-3\xi_1(2\xi_2+\xi_1), \quad  |Q_7'(\xi_2)|\gtrsim |\xi_1\xi_2|.
\ee
In addition, one also has
\[|L_1+L_2+L_3|= |-\xi_1 (2\xi_2^2+3\xi_1\xi_2+3\xi_1^2-1)|\gtrsim |\xi_1\xi_2\xi_3|,\]
which implies that $\la L_1 \ra \gtrsim \la \xi_1\xi_2\xi_3 \ra$ since $|\xi_1|, |\xi_2|>1$. Then, we can just follow the proof in \textbf{Case A3.1}.

  \item{\bf Case B3.2.} If $(\xi_1,\xi_2,\xi_3)$ lies in $\mb{(b)}$, that is, $\xi_1\geq 0$, $\xi_2\leq0$, $\xi_3\leq 0$.

      In this case, we  have
\begin{equation*}
  Q_8(\xi_2):=L_2+L_3=-2\xi_2^3-3\xi_1\xi_2^2-(3\xi_1^2-1)\xi_2-\tau_1
  -\xi_1^3  +\frac12\xi_1,
\end{equation*}
and
\be\label{Q-8}
Q_8'(\xi_2)=-6\xi_2^2-6\xi_1\xi_2-3\xi_1^2+1,
\ee
which leads to $|Q_8'(\xi_2)|\gtrsim |\xi_1\xi_2|$ for $|\xi_1|$, $|\xi_2|>1$.
In addition, one also has
\[|L_1+L_2+L_3|= |-(\xi_1+\xi_2)(2\xi_2^2+\xi_1\xi_2+2\xi_1^2-1)|\gtrsim |\xi_1\xi_2\xi_3|,\]
which implies that $\la L_1 \ra \gtrsim \la \xi_1\xi_2\xi_3 \ra$ since $|\xi_1|, |\xi_2|>1$. To bound \eqref{L21}, we can just follow the proof in \textbf{Case A3.1}.
  \item{\bf Case B3.3.} If $(\xi_1,\xi_2,\xi_3)$ lies in $\mb{(c)}$, that is, $\xi_1\leq 0$, $\xi_2\geq0$, $\xi_3\leq 0$.

      In this case, we  have
\begin{equation*}
  Q_9(\xi_2):=L_2+L_3=-3\xi_1\xi_2(\xi_1+\xi_2)-\tau_1
  -\xi_1^3  +\frac12\xi_1,
\end{equation*}
and
\be\label{Q-9}
Q_9'(\xi_2)=-3\xi_1(2\xi_2+\xi_1).
\ee
In addition, one also has
\[|L_1+L_2+L_3|= |-3\xi_1\xi_2(\xi_1+\xi_2)|\gtrsim |\xi_1\xi_2\xi_3|,\]
which implies that $\la L_1 \ra \gtrsim \la \xi_1\xi_2\xi_3 \ra$. To bound \eqref{L21}, we can just follow the proof in \textbf{Case A3.3}.
\end{itemize}
\vskip\baselineskip
\begin{itemize}
  \item{\bf Case B4.} $|\xi_1|$, $|\xi_2|>1$ and $\la L_2\ra=\mbox{MAX}$.
      It follows that
          \[\la L_1\ra^\frac12 \la L_2\ra^\frac12 \la L_3\ra^{1-\si}\geq \la L_2\ra^{2-3\si} \la L_1\ra^\si \la L_3\ra^{\si}.\]
      Then, similar to the \textbf{Case A3}, it suffices to bound
      \begin{equation}\label{L31}
\sup_{\xi_2,\tau_2}\frac{\la \xi_2\ra^{2\rho}}{\la L_2\ra^{4-6\si}}\int\frac{|\xi_3|^2\la \xi_1\ra^{2\rho}}{\la \xi_3\ra^{2\rho}\la L_1+L_3\ra^{2\si}}  d\xi_1,
\end{equation}
with considering $(\xi_2,\tau_2)$ fixed in the integral \eqref{L31} and
\begin{align*}
L_1+L_3=&\tau_2-|\xi_2|^3+\frac12|\xi_2|+\tau_3-|\xi_3|^3+
\frac12|\xi_3|\nonumber\\
=&- |\xi_2|^3+\frac12|\xi_2|-|\xi_3|^3+\frac12|\xi_3|-\tau_1.
\end{align*}

  \item{\bf Case B4.1.} If $(\xi_1,\xi_2,\xi_3)$ lies in $\mb{(a)}$, that is, $\xi_1\geq 0$, $\xi_2\geq0$, $\xi_3\leq 0$.

In this case, we  have
\begin{equation*}
  Q_{10}(\xi_1):=L_1+L_3=-2\xi_1^3 -3\xi_2\xi_1^2-(3\xi_2^2-1)\xi_1-\tau_2
  -\xi_2^3  +\frac12\xi_2,
\end{equation*}
and
\be\label{Q-10}
Q_{10}'(\xi_1)=-6\xi_1^2-6\xi_1\xi_2-3\xi_2^2+1,
\ee
with $|Q_{10}'(\xi_1)|\gtrsim |\xi_1\xi_2|$ since $|\xi_1|$, $|\xi_2|>1$.
In addition, one also has
\[|L_1+L_2+L_3|= |-2\xi_1^3 -3\xi_2\xi_1^2-3\xi_2^2\xi_1+\xi_1|\gtrsim |\xi_1\xi_2\xi_3|,\]
which implies that $\la L_1 \ra \gtrsim \la \xi_1\xi_2\xi_3 \ra$ for $|\xi_1|$, $|\xi_2|>1$. To bound \eqref{L31}, we can just follow the proof in \textbf{Case A3.1}.

  \item{\bf Case B4.2.} If $(\xi_1,\xi_2,\xi_3)$ lies in $\mb{(b)}$, that is, $\xi_1\geq 0$, $\xi_2\leq0$, $\xi_3\leq 0$.

      In this case, we  have
\begin{equation*}
  Q_{11}(\xi_1):=L_1+L_3=-2\xi_1^3 -3\xi_2\xi_1^2-(3\xi_2^2-1)\xi_1-\tau_2
  -\xi_2^3  +\frac12\xi_2,
\end{equation*}
and
\be\label{Q-11}
Q_{11}'(\xi_1)=-6\xi_1^2-6\xi_1\xi_2-3\xi_2^2+1,
\ee
with $|Q_{11}'(\xi_1)|\gtrsim |\xi_1\xi_2|$ since $|\xi_1|$, $|\xi_2|>1$.
In addition, one also has
\[|L_1+L_2+L_3|= |-(\xi_1+\xi_2)(2\xi_1^2+\xi_1\xi_2+2\xi_2^2-1)|\gtrsim |\xi_1\xi_2\xi_3|,\]
which implies that $\la L_1 \ra \gtrsim \la \xi_1\xi_2\xi_3 \ra$ for $|\xi_1|$, $|\xi_2|>1$. To bound \eqref{L31}, we can just follow the proof in \textbf{Case A3.1}.
  \item{\bf Case B4.3.} If $(\xi_1,\xi_2,\xi_3)$ lies in $\mb{(c)}$, that is, $\xi_1\leq 0$, $\xi_2\geq0$, $\xi_3\leq 0$.

        In this case, we  have
\begin{equation*}
  Q_{12}(\xi_1):=L_1+L_3=-3\xi_1\xi_2(\xi_1+\xi_2)-\tau_2
  -\xi_2^3  +\frac12\xi_2,
\end{equation*}
and
\be\label{Q-12}
Q_{12}'(\xi_1)=-3\xi_2(2\xi_1+\xi_2).
\ee
In addition, one also has
\[|L_1+L_2+L_3|=| -3\xi_1\xi_2(\xi_1+\xi_2)|\gtrsim |\xi_1\xi_2\xi_3|,\]
which implies that $\la L_1 \ra \gtrsim \la \xi_1\xi_2\xi_3 \ra$. To bound \eqref{L31}, we can just follow the proof in \textbf{Case A3.3}.
      \end{itemize}
      \vskip\baselineskip
      \begin{itemize}
        \item {\bf Case B5.} $|\xi_1|$, $|\xi_2|>1$ and $\la L_3\ra=\mbox{MAX}$.\\
             Similar to the proof in \textbf{Case A5}, we have
       \[\la L_1\ra^\frac12 \la L_2\ra^\frac12 \la L_3\ra^{1-\si}\geq \la L_3\ra^{2-3\si} \la L_1\ra^\si \la L_2\ra^{\si}.\]
Hence, it suffices to show the bound for
\begin{equation}\label{L41}
\sup_{\xi_3,\tau_3}\frac{|\xi_3|^2}{\la \xi_3\ra^{2\rho}\la L_3\ra^{4-6\si}}\int\frac{\la \xi_1\ra^{2\rho}\la \xi_2\ra^{2\rho}}{ \la L_1+L_2\ra^{2\si}}  d\xi_1,
\end{equation}
with considering $(\xi_3,\tau_3)$ fixed in the integral \eqref{L41} and
\begin{align*}
L_1+L_2=&\tau_1-|\xi_1|^3+\frac12|\xi_1|+\tau_2+|\xi_2|^3-\frac12|\xi_2|\nonumber\\
=& -|\xi_1|^3+\frac12|\xi_1|+|\xi_2|^3-\frac12|\xi_2+\tau_3.
\end{align*}
        \item {\bf Case B5.1.} If $(\xi_1,\xi_2,\xi_3)$ lies in $\mb{(a)}$, that is, $\xi_1\geq 0$, $\xi_2\geq0$, $\xi_3\leq 0$.

        In this case, we  have
\begin{equation*}
  Q_{13}(\xi_1):=L_1+L_2=-2\xi_1^3 -3\xi_3\xi_1^2-(3\xi_3^2-1)\xi_1-\tau_3
  -\xi_3^3  +\frac12\xi_3,
\end{equation*}
and
\be\label{Q-13}
Q_{13}'(\xi_1)=-6\xi_1^2-6\xi_1\xi_3-3\xi_3^2+1,
\ee
with $|Q_{13}'(\xi_1)|\gtrsim |\xi_1\xi_2|$ since  $|\xi_1|>1$.
In addition, one also has
\[|L_1+L_2+L_3|= |-2\xi_1^3 -3\xi_3\xi_1^2-(3\xi_3^2-1)\xi_1|\gtrsim |\xi_1\xi_2\xi_3|,\]
which implies that $\la L_1 \ra \gtrsim \la \xi_1\xi_2\xi_3 \ra$. To bound \eqref{L41}, we can just repeat the proof in \textbf{Case A3.1}.
        \item {\bf Case B5.2.} If $(\xi_1,\xi_2,\xi_3)$ lies in $\mb{(b)}$, that is, $\xi_1\geq 0$, $\xi_2\leq0$, $\xi_3\leq 0$.

        In this case, we  have
\begin{equation*}
  Q_{14}(\xi_1):=L_1+L_2=3\xi_1\xi_3(\xi_1+\xi_3)-\tau_3
  +\xi_3^3  -\frac12\xi_3,
\end{equation*}
and
\be\label{Q-14}
Q_{14}'(\xi_1)=3\xi_3(2\xi_1+\xi_3).
\ee
In addition, one also has
\[|L_1+L_2+L_3|=| \xi_3(2\xi_3^2+3\xi_1\xi_3+3\xi_1^2-1)|\gtrsim |\xi_1\xi_2\xi_3|,\]
which implies that $\la L_1 \ra \gtrsim \la \xi_1\xi_2\xi_3 \ra$ for $|\xi_1|>1$. To bound \eqref{L41}, we can just repeat the proof in \textbf{Case A5.1}.
       \item{\bf Case B5.3.} If $(\xi_1,\xi_2,\xi_3)$ lies in $\mb{(c)}$, that is, $\xi_1\leq 0$, $\xi_2\geq0$, $\xi_3\leq 0$.

        In this case, we  have
\begin{equation*}
  Q_{15}(\xi_1):=L_1+L_2=-3\xi_1\xi_3(\xi_1+\xi_3)-\tau_3
  -\xi_3^3  +\frac12\xi_3,
\end{equation*}
and
\be\label{Q-15}
Q_{15}'(\xi_1)=-3\xi_3(2\xi_1+\xi_3).
\ee
In addition, one also has
\[|L_1+L_2+L_3|= |-3\xi_1\xi_3(\xi_1+\xi_3)|\gtrsim |\xi_1\xi_2\xi_3|,\]
which implies that $\la L_1 \ra \gtrsim \la \xi_1\xi_2\xi_3 \ra$. To bound \eqref{L41}, we can just repeat the proof in \textbf{Case A5.1}.
      \end{itemize}
      The proof is now complete.
 \end{proof}

\section{ Local Well-posedness}
\setcounter{equation}{0}

Before stating the conclusion for the local well-posedness of IBVP \eqref{hsb}, let us define
\[\X_\a^{s,\frac12,\si}:= C(\R; H^s(\R))\cap X_\a^{s,\frac12,\si}, \]
with the norm
\[\|w\|_{\X_\a^{s,\frac12,\si}}=\l(\sup_{t\in \R} \|w(\cdot,t)\|^2_{H^s(\R)}+\|w\|_{X_\a^{s,\frac12,\si}}\r)^{\frac12}.\]
 We also denote $X^{s,\frac12,\si}_{\a,+}:=X^{s,\frac12,\si}_\a(\R^+\times \R^+ )$ and  $\X^{s,\frac12,\si}_{\a,+}:=\X_\a^{s,\frac12,\si}(\R^+\times \R^+ )$.
 \begin{thm}\label{conditional1}
 Let $-\frac34<s\leq -\frac12$   be given. There exist $\eps=\eps(s)>0$ and $ \sigma_0=\sigma_0(s)\in (\frac12,1)$ such that if
 \[(\varphi,\psi,\vec{h})\in H^s(\R^+) \times H^{s-1}(\R^+)\times \H^s(\R^+)\]
 and
 \be\label{smallness of initial data}
 |\a |\leq \eps,  \ \ \|(\varphi, \psi)\|_{H^s(\R^+)\times H^{s-1}(\R^+)}\leq \eps, \\ \|\vec{h}\|_{\H^s(\R^+)}\leq \eps,\ee
  then the IBVP \eqref{hsb} admits a unique solution
 $u\in \X_{\a,+}^{s,\frac12,\si}$
 for any $ \sigma\in(\frac12,\sigma_0] $ in an interval with size $1$
and the corresponding solution map is real analytic.
 \end{thm}

 \begin{proof}
  Let $\eps\in(0,1)$ and $ \sigma_0\in(\frac12,1) $ be constants which will be determined later.  For given 
  \[(\varphi,\psi,\vec{h})\in Y_s:= H^s(\R^+) \times H^{s-1}(\R^+)\times \H^s(\R^+),\]
  with
 \[
 \|(\varphi, \psi)\|_{H^s(\R^+)\times H^{s-1}(\R^+)}\leq \eps, \ \ \|\vec{h}\|_{\H^s(\R^+)}\leq \eps,\] 
 and for any $ \sigma\in(\frac12,\sigma_0] $, let
\begin{equation*}
    S_{C,\eps}=\l\{u\in \X^{s,\frac12,\si}_{\a,+}, \|u\|_{\X^{s,\frac12,\si}_{\a,+}}\leq CE_0\r\},
\end{equation*}
where
\be\label{small E_0}
E_0:=\|(\varphi, \psi)\|_{H^s(\R^+)\times H^{s-1}(\R^+)}+\|\vec{h}\|_{\H^s(\R^+)}\leq 2\eps.\ee
Then the set $S_{C,\eps}$ is a convex, closed and bounded subset of $\X^{s,\frac12,\si}_{\a,+}$.
Given $u\in S_{C,\eps}$, we define
\begin{align*}
\Gamma(u)(x,t)=\eta(t)W_R(\varphi^*,\psi^*)&+\eta(t)W_{bdr}(\vec{h}-\vec{p})\\&
+\eta(t)
\left(\int^t_0 [W_R(0,f)](x,t-t')dt'-W_{bdr}(\vec{q})\right),
\end{align*}
where $f= u^2$, $\vec{p}$   and
$\vec{q}$  are as defined in Proposition \ref{represent}.
It  can then  be verified
that $\Gamma(u)(x,t)$ solves the IBVP in $[0,1]$.
We  will show that $\Gamma$ is a contraction map from $S_{C,\eps}$ to $S_{C,\eps}$ for proper $C$ and $\eps$.  Applying Lemma \ref{R} and \ref{lemma3},  Proposition \ref{kato}   leads to
\begin{align*}
\|\Gamma(u)\|_{H^s(\R^+)}\lesssim &\|\eta(t)W_R(\varphi^*,\psi^*)\|_{H_x^s(\R^+)}+\|\eta(t)W_{bdr}(\vec{h}-\vec{p})\|_{H_x^s(\R^+)}\\
&+\left\|\eta (t)
\left(\int^t_0 [W_R(0,f)](x,t-t')dt'-W_{bdr}(\vec{q})\right)\right\|_{H_x^s(\R^+)}\\
\lesssim & \|(\varphi,\psi,\vec{h})\|_{Y_s}\\
&+\left\|\eta(t)
\left(\int^t_0 [W_R(0,f)](x,t-t)dt'\right)\right\|_{X_{\a}^{s,\sigma}} + \left\|   \eta(t)W_{bdr}(\vec{q})  \right\|_{H^{s}_{x}(\m{R}^+)}.
\end{align*}
According to Lemma \ref{for}, Proposition \ref{kato} and Theorem \ref{bilinear}, there exsits some $ \sigma_0=\sigma_0(s)\in(\frac12,1) $ such that for any $\si\in(\frac12,\si_0]$,
\begin{align*}
\left\|\eta(t)\int^t_0 [W_R(0,f)](x,t-t')dt'\right\|_{X_{\a}^{s,\sigma}}\lesssim & \left\|\l(\frac{\xi^2\widehat{f}}{\phi(\xi)}\r)^\vee\right\|_{X_{\a}^{s,\si-1}}\\
\lesssim  & \left\|\l(\frac{\xi^2\widehat{f}}{\phi(\xi)}\r)^\vee\right\|_{X_{\a}^{s+2,\si-1}}
\leq C_1\|u\|^2_{X_{\a,+}^{s,\frac12,\si}}.
\end{align*}
 Moreover, according to Lemma \ref{lemma3}, Proposition \ref{kato}, and  Theorem \ref{bilinear},
\begin{align*}
 \left\|   \eta(t)W_{bdr}(\vec{q})  \right\|_{H^{s}_{x}(\m{R}^+)} \lesssim  \|\vec{q}\|_{{\mathcal H}^s} &\lesssim \left\|\l(\frac{\xi^2\widehat{f}}{\phi(\xi)}\r)^\vee\right\|_{X_{\a}^{s+2,\si-1}}
\leq  C_1\|u\|^2_{X_{\a,+}^{s,\frac12,\si}}.
\end{align*}
Hence,
\[\|\Gamma(u)\|_{H^s(\R^+)}\leq C_1\|(\varphi,\psi,\vec{h})\|_{Y_s}+ 2C_1 \|u\|^2_{X_{\a,+}^{s,\frac12,\si}}.\]
Furthermore, we have,
\begin{align*}
\|\Gamma(u)\|_{X_{\a,+}^{s,\frac12,\si}}\lesssim &\|\eta(t)W_R(\varphi^*,\psi^*)\|_{X_{\a}^{s,\frac12,\si}}+
\|\eta(t)W_{bdr}(\vec{h}-\vec{p})\|_{X_{\a}^{s,\frac12,\si}}\\
&+\left\|\eta (t)
\left(\int^t_0 [W_R(0,f)](x,t-t')dt'-W_{bdr}(\vec{q})\right)\right\|_{X_{\a}^{s,\frac12,\si}}\\
\lesssim & \|(\varphi,\psi,\vec{h})\|_{Y_s}\\
&+\left\|\eta (t)
\left(\int^t_0 [W_R(0,f)](x,t-t')dt'-W_{bdr}(\vec{q})\right)\right\|_{X_{\a}^{s,\frac12,\si}}.
\end{align*}
Again, according to Lemma \ref{for}, \ref{Lemma, lin est} and \ref{lemma3}, Proposition \ref{kato}, and  Theorem \ref{bilinear}, one has,
\begin{align*}
\left\|\eta(t)
\int^t_0 [W_R(0,f)](x,t-\tau)d\tau \right\|_{X_{\a}^{s,\frac12,\si}}\lesssim & \left\|\l(\frac{\xi^2\widehat{f}}{\phi(\xi)}\r)^\vee\right\|_{X_{\a}^{s,a}}
\leq C_1\|u\|^2_{X^{s,\frac12,\si}_{\a,+}},
\end{align*}
and
\[
\left\|\eta (t)[W_{bdr}(\vec{q})](x,t)\right\|_{X_{\a}^{s,\frac12,\si}} \lesssim  \|\vec{q}\|_{{\mathcal H}^s} \leq C_1\|u\|^2_{X^{s,\frac12,\si}_{\a,+}}.
\]
Thus,
\[\|\Gamma(u)\|_{X_{\a,+}^{s,\frac12,\si}}\leq C_1\|(\varphi,\psi,\vec{h})\|_{Y_s}+2 C_1 \|u\|^2_{X^{s,\frac12,\si}_{\a+}}.\]
It follows that
\[\|\Gamma(u)\|_{\X^{s,\frac12,\si}_{\a,+}}\leq C_1 E_0 + 2C_1 C^2 E^2_0.\]
By choosing $C=2C_1$ and $\eps=\frac{1}{64 C_1^2}$,  it then follows from (\ref{small E_0}) that
\[\|\Gamma(u)\|_{\X^{s,\frac12,\si}_{\a,+}}\leq CE_0.\]
Similar argument can be drawn to establish  $$\|\Gamma(u)-\Gamma(v)\|_{\X^{s,\frac12,\si}_{\a,+}}\leq \frac12 \|u-v\|_{\X^{s,\frac12,\si}_{\a,+}}, $$
 with $u,v\in {X^{s,\frac12,\si}_{\a,+}}$. Hence, the map $\Gamma$ is a contraction.  Because $\eta\equiv 1$ on $(0,1)$, $u$ is the unique solution to the IBVP \eqref{hsb} on a time of size $1$. Therefore, we conclude the proof for Theorem \ref{conditional}.
 \end{proof}

%
%

{\small
\begin{thebibliography}{10}

\bibitem{1}
I.~Bejenaru and T.~Tao.
\newblock Sharp well-posedness and ill-posedness results for a quadratic
non-linear schrodinger equation.
\newblock {\em Journal of Functional Analysis}, 233(1):228--259, 2006.

\bibitem{17}
J.~L. Bona, S.~M. Sun, and B.-Y. Zhang.
\newblock A non-homogeneous boundary-value problem for the korteweg-de vries
equation in a quarter plane.
\newblock {\em Transactions of the American Mathematical Society},
354(2):427--490, 2002.

\bibitem{bonauncondition}
J.~L. Bona, S.-M. Sun, and B.-Y. Zhang.
\newblock Conditional and unconditional well-posedness for nonlinear evolution
equations.
\newblock {\em Advances in Differential Equations}, 9(3-4):241--265, 2004.

\bibitem{15}
J.~L. Bona, S.~M. Sun, and B.-Y. Zhang.
\newblock Boundary smoothing properties of the Korteweg-de Vries equation in a
quarter plane and applications.
\newblock {\em Dynamics of Partial Differential Equations}, 3(1):1--69, 2006.

\bibitem{14}
J.~L. Bona, S.~M. Sun, and B.-Y. Zhang.
\newblock Non-homogeneous boundary value problems for the Korteweg-de Vries and
the Korteweg-de Vries-Burgers equations in a quarter plane.
\newblock {\em Annales de l'Institut Henri Poincar\'e: Analyse Non Lin\'eaire},
25(6):1145--1185, 2008.

\bibitem{89}
J.~L. Bona, S.-M. Sun, and B.-Y. Zhang.
\newblock Non-homogeneous boundary-value problems for one-dimensional nonlinear
schr{\"o}dinger equations.
\newblock {\em Journal de Math{\'e}matiques Pures et Appliqu{\'e}es},
109:1--66, 2018.

\bibitem{cava}
M.~Cavalcante.
\newblock The initial-boundary-value-problem for some quadratic nonlinear
{S}chr\"odinger equations on the half-line.
\newblock {\em Differential and Integral Equations}, 30:521--554, 2016.

\bibitem{28}
C.~Christov and M.~Velarde.
\newblock Well-posed Boussinesq paradigm with purely spatial higher-order
derivatives.
\newblock {\em Physical Review E}, 54(2), 1996.

\bibitem{tao2}
J.~Colliander, M.~Keel, G.~Staffilani, H.~Takaoka, and T.~Tao.
\newblock Sharp global well-posedness for {K}d{V} and modified {K}d{V} on
{$\Bbb R$} and {$\Bbb T$}.
\newblock {\em Journal of the American Mathematical Society}, 16(3):705--749, 2003.

\bibitem{31}
J.~E. Colliander and C.~E. Kenig.
\newblock The generalized Korteweg-de Vries equation on the half line.
\newblock {\em Communications in Partial Differential Equations},
27(11-12):2187--2266, 2002.

\bibitem{97}
E.~Compaan and N.~Tzirakis.
\newblock Well-posedness and nonlinear smoothing for the “good” boussinesq
equation on the half-line.
\newblock {\em Journal of Differential Equations}, 262(12):5824--5859, 2017.

\bibitem{95}
M.~Erdo{\u{g}}an and N.~Tzirakis.
\newblock Regularity properties of the cubic nonlinear Schr{\"o}dinger equation
on the half line.
\newblock {\em Journal of Functional Analysis}, 271(9):2539--2568, 2016.

\bibitem{34}
A.~Esfahani and L.~G. Farah.
\newblock Local well-posedness for the sixth-order Boussinesq equation.
\newblock {\em Journal of Mathematical Analysis and Applications},
385(1):230--242, 2012.

\bibitem{35}
A.~Esfahani, L.~G. Farah, and H.~Wang.
\newblock Global existence and blow-up for the generalized sixth-order
Boussinesq equation.
\newblock {\em Nonlinear Analysis: Theory, Methods and Applications},
75(11):4325--4338, 2012.

\bibitem{37}
A.~Esfahani and H.~Wang.
\newblock A bilinear estimate with application to the sixth-order Boussinesq
equation.
\newblock {\em Differential and Integral Equations}, 27(5/6):401--414, 05 2014.

\bibitem{41}
L.~G. Farah.
\newblock Local solutions in Sobolev spaces with negative indices for the
``good'' Boussinesq equation.
\newblock {\em Communications in Partial Differential Equations},
34(1-3):52--73, 2009.

\bibitem{GTV97}
J.~Ginibre, Y.~Tsutsumi and G.~Velo.
\newblock On the {C}auchy problem for the {Z}akharov system.
\newblock{\em Journal of Functional Analysis}, 151(2): 384--436, 1997.

%

\bibitem{holmer}
J.~Holmer.
\newblock The initial-boundary value problem for the 1D nonlinear
{S}chr\"odinger equation on the half-line.
\newblock {\em Differential and Integral Equations}, 18:647--668, 2005.

\bibitem{47}
J.~Holmer.
\newblock The initial-boundary value problem for the Korteweg-de Vries
equation.
\newblock {\em Communications in Partial Differential Equations},
31(7-9):1151--1190, 2006.
%

\bibitem{54}
C.~E. Kenig, G.~Ponce, and L.~Vega.
\newblock Well-posedness of the initial value problem for the Korteweg-de Vries
equation.
\newblock {\em Journal of the American Mathematical Society}, 4(2):323--347,
1991.

\bibitem{53}
C.~E. Kenig, G.~Ponce, and L.~Vega.
\newblock The Cauchy problem for the Korteweg-de Vries equation in Sobolev
spaces of negative indices.
\newblock {\em Duke Mathematical Journal}, 71(1):1--21, 1993.

\bibitem{52}
C.~E. Kenig, G.~Ponce, and L.~Vega.
\newblock Well-posedness and scattering results for the generalized Korteweg-de
Vries equation via the contraction principle.
\newblock {\em Communications on Pure and Applied Mathematics}, 46(4):527--620,
1993.

\bibitem{50}
C.~E. Kenig, G.~Ponce, and L.~Vega.
\newblock A bilinear estimate with applications to the KdV equation.
\newblock {\em Journal of the American Mathematical Society}, 9(2):573--603,
1996.

\bibitem{51}
C.~E. Kenig, G.~Ponce, and L.~Vega.
\newblock Quadratic forms for the $1$-d semilinear Schr\"odinger equation.
\newblock {\em Transactions of the American Mathematical Society},
348(8):3323--3353, 1996.

\bibitem{kdv-1}
R.~Killip and M.~Vi\c{s}an.
\newblock Kd{V} is well-posed in {$H^{-1}$}.
\newblock {\em Annals of Mathematics (2)}, 190(1):249--305, 2019.

\bibitem{55}
N.~Kishimoto and K.~Tsugawa.
\newblock Local well-posedness for quadratic nonlinear Schr\"odinger equations
and the ``good'' Boussinesq equation.
\newblock {\em Differential and Integral Equations}, 23(5-6):463--493, 2010.

\bibitem{96}
S.~Li, M.~Chen, and B.-Y.  Zhang.
\newblock A non-homogeneous boundary value problem of the sixth order
Boussinesq equation in a quarter plane.
\newblock {\em Discrete and Continuous Dynamical Systems - Series A}, 38(5),
2018.

\bibitem{96-1}
S.~Li, M.~Chen, and B.-Y.  Zhang.
\newblock A non-homogeneous boundary value problem of a higher order Boussinesq
equation in a quarter plane.
\newblock {\em Journal of Mathematical Analysis and Applications}, 492:124406, 2020.

\bibitem{96-2}
S.~Li, X.~Yang, and B.-Y.  Zhang.
\newblock Non-homogeneous boundary value problem for a KdV-KdV system.
\newblock {\em Preprint}, 2021.

\bibitem{59}
F.~Linares.
\newblock Global existence of small solutions for a generalized Boussinesq
equation.
\newblock {\em Journal of Differential Equations}, 106(2):257--293, 1993.

\bibitem{66}
G.~A. Maugin.
\newblock {\em Nonlinear waves in elastic crystals}.
\newblock Oxford University Press, Oxford, 1999.

\bibitem{tao1}
T.~Tao.
\newblock Multilinear weighted convolution of {$L^2$}-functions, and
applications to nonlinear dispersive equations.
\newblock {\em American Journal of Mathematics}, 123(5):839--908, 2001.

	\bibitem{Tao06}
T.~Tao.
\newblock {\em Nonlinear dispersive equations: local and global analysis}. Volume 106 of {\em CBMS
	Regional Conference Series in Mathematics}.
\newblock American Mathematical Society, Providence, RI, 2006.

\bibitem{xin}
X.~Yang and B.-Y. Zhang.
\newblock Local well-posedness of the coupled KdV-KdV systems on the whole line
$\R$.
\newblock {\em arXiv, https://arxiv.org/pdf/1812.08261.pdf}, 2018.

\bibitem{zhang1}
B.-Y.  Zhang.
\newblock Taylor series expansion for solutions of the Korteweg-de Vries
equation with respect to their initial values.
\newblock {\em Journal of Functional Analysis}, 1995.

\bibitem{88}
B.-Y.  Zhang.
\newblock Analyticity of solutions of the generalized Korteweg-de Vries equation
with respect to their initial values.
\newblock {\em SIAM Journal on Mathematical Analysis}, 26(6):1488--1513, 1995.

\end{thebibliography}
}

\end{document}